\documentclass[12pt,twoside]{amsart}
\usepackage{amsfonts}
\usepackage{epsfig,graphics}
\usepackage{amssymb}
\usepackage{amscd}

\newtheorem{theorem}{Theorem}[section]

\newtheorem{proposition}[theorem]{Proposition}
       
\newtheorem{definition}[theorem]{Definition}

\newtheorem{lemma}[theorem]{Lemma}

\newtheorem{corollary}[theorem]{Corollary}

\newtheorem{remark}[theorem]{Remark}

\newtheorem{question}{Question}

\newcommand{\rk}{\text{rk}}

\newcommand{\ad}{\text{ad}}

\newcommand{\tr}{\text{tr}}

\newcommand{\SL}{\text{SL}}
\newcommand{\GL}{\text{GL}}

\newcommand{\id}{\text{Id}}
\newcommand{\im}{\mathrm{im}}

\newcommand{\Isom}{\text{Isom}}

\newcommand{\Ad}{\text{Ad}}
\newcommand{\Aut}{\text{Aut}}

\newcommand{\BN}{{\bf N}} 
\newcommand{\BR}{{\bf R}}
\newcommand{\BC}{{\bf C}}
\newcommand{\BH}{{\bf H}}

\newcommand{\lieg}{{\mathfrak{g}}}

\newcommand{\liep}{{\mathfrak{p}}}

\newcommand{\alggl}{{\mathfrak{g}\mathfrak{l}}}
\newcommand{\algsl}{{\mathfrak{s}\mathfrak{l}}}

\newcommand{\alginf}{{\mathfrak{i}\mathfrak{n}\mathfrak{f}}}

\newenvironment{Pf}{\medskip \noindent {\bf Proof: }}
   {$\diamondsuit$ }

\let\Cal\mathcal 
\let\Bbb\mathbb
\let\frak\mathfrak
\let\phi\varphi

\newcommand{\al}{\alpha}

\newcommand{\om}{\omega}
\newcommand{\ka}{\kappa}
\newcommand{\ph}{\varphi}
\newcommand{\la}{\lambda}
\newcommand{\Ga}{\Gamma}
\newcommand{\La}{\Lambda}
\newcommand{\Om}{\Omega}
\newcommand{\x}{\times}
\renewcommand{\o}{\circ}

\begin{document}
\pagenumbering{arabic} \title{Essential Killing fields of parabolic
  geometries} \author{Andreas \v Cap and Karin Melnick} \date{\today}
\thanks{This project was initiated during the second author's Junior
  Research Fellowship at the Erwin Schr\"odinger Institute in Vienna,
  and was continued during a workshop at the ESI, ``Cartan
  connections, geometry of homogeneous spaces, and dynamics.''  \v Cap
  is supported by Fonds zur F\"orderung der wissenschaftlichen
  Forschung, project P23244-N13, and Melnick is supported by NSF grant
  DMS-1007136}

\begin{abstract}
We study vector fields generating a local flow by automorphisms of a
parabolic geometry with \emph{higher order fixed points}.  We develop
general tools extending the techniques of \cite{nagano.ochiai.proj},
\cite{frances.riemannien}, and \cite{fm.champsconfs}, and we apply
them to almost Grassmannian, almost quaternionic, and contact
parabolic geometries, including CR structures.  We obtain descriptions
of the possible dynamics of such flows near the fixed point and strong
restrictions on the curvature; in some cases, we can show vanishing of
the curvature on a nonempty open set.  Deriving consequences for a
specific geometry entails evaluating purely algebraic and
representation-theoretic criteria in the model homogeneous space.
\end{abstract}
\maketitle

\emph{Dedicated to Michael Eastwood on the occasion of his 60th birthday.}

\section{Introduction}

An array of results in differential geometry tell us that geometric
structures admitting a large group of automorphisms are special and
must have a particularly simple form.  For example, a Riemannian
manifold $M^n$ with $\Isom(M)$ of maximum possible dimension
$\tfrac{n(n+1)}{2}$ must have constant sectional curvature and thus be
a space form.  More generally, the maximal dimension for the Lie algebra of
Killing vector fields on a Riemannian manifold, or for the Lie algebra
of infinitesimal automorphisms for many classical geometric structures,
can be only attained on open subsets of a homogeneous model.

 
In some cases, the existence of a \textit{single} automorphism or
infinitesimal automorphism of special type restricts the geometry. 
Special automorphisms that exists for some geometric structures
are those that equal the identity to first order at a point; note that
because of the exponential map, Riemannian metrics never admit such
automorphisms, except the identity.  The projective transformations of
projective space ${\bf RP}^n$, on the other hand, do include such
automorphisms: there is $\id\neq g \in \Aut({\bf RP}^n)$ with $g.x =
x$ and $Dg_x = \id$. The space ${\bf RP}^n$ viewed as a homogeneous
space of the group of projective transformations is the model for
\textit{classical projective structures}. Such a structure on a
manifold $M$ is an equivalence class $[ \nabla ]$ of torsion-free
linear connections on $TM$ having the same sets of geodesics up to
reparametrization. An automorphism is a diffeomorphism of $M$
preserving $[\nabla]$, or, equivalently, preserving the corresponding
family of geodesic paths as unparametrized curves.

For $M$ connected, an automorphism of a classical projective structure
on $M$ is uniquely determined by its two--jet at a single point.
Non--trivial automorphisms fixing a point to first order are examples
of \emph{essential} automorphisms---ones not preserving any connection
in the projective class $[\nabla]$.  Nagano and Ochiai
\cite{nagano.ochiai.proj} proved that if a compact, connected manifold
$M^n$ with a torsion-free connection admits a nontrivial vector field
for which the flow is projective and trivial to first order at a point
$x_0$, then $M$ is projectively \emph{flat} on a neighborhood of
$x_0$---that is, locally projectively equivalent to ${\bf RP}^n$.

Pseudo--Riemannian conformal structures may also admit non--trivial
automorphisms which equal the identity to first order in a point.  In
this case, Frances and the second author prove analogous results in
\cite{frances.riemannien} and \cite{fm.champsconfs}.  Their theorems
say that if a conformal vector field $X$ vanishes at a point $x$, and
if the flow $\{ \varphi^t_X \}_{t \in \BR}$ is unbounded but has
precompact differential at $x$, then the manifold is conformally flat
on a nonempty open set $U$ with $x \in \overline{U}$.

Both proofs make use of the Cartan geometry canonically associated to
the structures in question and of the contracting---though not
necessarily uniformly contracting---dynamics of the given flows.  In
both cases, the Cartan geometry is a \emph{parabolic geometry}, one
for which the homogeneous model is $G/P$ for G a semisimple Lie group
and $P$ a parabolic subgroup.  An introduction to the general theory
of parabolic geometries can be found in \cite{cap.slovak.book.vol1}.
See \cite{cap.srni04} and \cite{cap.deformations} for general results
on automorphisms and infinitesimal automorphisms.

In this article, we develop machinery to apply these ideas to study the 
behavior of a certain class of flows fixing a point that include the projective and conformal flows described above, in the general setting of parabolic geometries.
Our results lead to descriptions of the possible dynamics of such flows near the fixed point and to strong restrictions on the curvature.  In some cases, we can show vanishing of the curvature on a nonempty open set.  Deriving consequences for a specific geometry entails evaluating purely algebraic and representation-theoretic criteria for the pair $(G,P)$.  

\subsection{Background}

\subsubsection{Cartan geometries of parabolic type}

Let $G$ be a semisimple Lie group with Lie algebra $\lieg$.  A
parabolic subalgebra of $\lieg$ can be specified by a $|k|$--grading
for some positive integer $k$, which is a grading of $\lieg$ of the
form $\lieg=\lieg_{-k}\oplus\dots\oplus\lieg_k$ such that no simple
ideal is contained in the subalgebra $\lieg_0$, and such that
the subalgebra $\lieg_- =\oplus_{i<0}\lieg_i$ is generated by $\lieg_{-1}$.  The parabolic subalgebra determined by the grading is then $\frak p=\oplus_{i\geq
  0}\lieg_i$, and a parabolic subgroup $P < G$ is a subgroup with
Lie algebra $\frak p$.  It is a fact that
$$ 
N^0_G(\frak p) \leq P \leq N_G(\frak p)
$$ 
where $N_G(\frak p)$ is the normalizer and $N^0_G(\frak p)$ is its
connected component of the identity.  The center of $\lieg_0$ contains the \emph{grading element} $A$, for which each $\lieg_i$ is an eigenspace of $\ad(A)$ with eigenvalue $i$.
The $|k|$--gradings of a given
Lie algebra $\lieg$ correspond to subsets of the simple roots when
$\lieg$ is complex, and subsets of simple restricted roots for $\lieg$
real, associated to a choice of Cartan subalgebra (which is maximally
non--compact in the real case); see section 3.2 of
\cite{cap.slovak.book.vol1}.

Defining $\lieg^i = \oplus_{j \geq i} \lieg_j$ makes $\lieg$ into a
filtered Lie algebra such that $\liep = \lieg^0$. The parabolic
subgroup $P$ acts by filtration--preserving automorphisms
under the adjoint action.  The subgroup $G_0 < P$ preserving the grading of $\lieg$ has Lie algebra $\lieg_0$.
Denote $\liep_+ = \lieg^1$, and let $P_+ <
P$ be the corresponding subgroup; it is unipotent and normal in $P$,
and $\exp:\liep_+\to P_+$ is a diffeomorphism.  Then $G_0 \cong P / P_+$, and it is closed and reductive. 

\begin{definition}  
\label{def:cartan.geometry}
Let $G$ be a Lie group with Lie algebra $\lieg$ and $P$ a closed
subgroup.  A \emph{Cartan geometry} on $M$ modeled on the pair
$(\lieg,P)$ is a triple $(M,B,\omega)$, where
\begin{enumerate}
\item $\pi : B \rightarrow M$ is a principal $P$-bundle 
\item $\omega \in \Omega^1(B,\lieg)$ is the \emph{Cartan connection}, satisfying
\begin{enumerate}
\item for all $b \in B$, the restriction $\omega_b : T_bB \rightarrow
  \lieg$ is a linear isomorphism
\item for all $p \in P$, the pullback $R_p^* \omega = \Ad(p^{-1}) \circ \omega$
\item for all $X \in \liep$, if $\tilde{X}$ is the \emph{fundamental
  vector field} $\tilde{X}(b) = \left. \frac{\mathrm{d}}{\mathrm{d}t}
  \right|_0 b.e^{tX}$, then $\omega(\tilde{X}) = X$.
\end{enumerate}
\end{enumerate}
\end{definition}

The Cartan connection generalizes the left-invariant Maurer-Cartan
form $\omega_G$ on $G$.  The following curvature is a complete obstruction
to local isomorphism of $(M,B,\om)$ to the homogeneous model $(G/P,G, \omega_G)$.

\begin{definition}\label{def:curvature}
 The \emph{curvature} of the Cartan connection $\om$ is the two--form
 $K\in\Om^2(B,\frak g)$ given by
$$K(\xi,\eta)=d\om(\xi,\eta)+[\om(\xi),\om(\eta)]$$ 
\end{definition}

\begin{definition}  A \emph{parabolic geometry} on a manifold $M$ is a
  Cartan geometry on $M$ modeled on $(\lieg, P)$ for $G$ a semisimple
  Lie group and $P$ a parabolic subgroup.   
\end{definition}

The Cartan connection $\om$ gives rise to natural local charts on $B$
as follows. To each $A\in\frak g$ corresponds an $\omega$--constant
vector field $\tilde A\in\frak X(B)$, characterized by $\om(\tilde
A)\equiv A$.  Note that $\tilde A$ is the fundamental vector field if
$A\in\liep$.  For any $b\in B$, and sufficiently small $A \in \lieg$,
define $\exp(b,A)=\exp_b(A)$ to be the image of $b$ under the
time-one flow along $\tilde A$.  There is a neighborhood $U$ of $0 \in
\lieg$ on which $\exp_b$ is defined and a diffeomorphism onto an open
subset of $B$. Composing the projection $\pi$ with the restriction of
$\exp_b$ to any linear subspace in $\lieg$ complementary to $\liep$,
we obtain a local chart on $M$.  The exponential map gives rise to a
notion of distinguished curves and to normal coordinates on a
parabolic geometry:

\begin{definition}\label{def:curves}
Consider a parabolic geometry on $M$ of type $(\lieg,P)$.
\begin{itemize}
\item For $X \in \lieg$, an \emph{exponential curve} in $M$ is the projection to $M$ of a curve $t\mapsto\exp(b,tX)$ for some $b \in B$. It is a \emph{distinguished curve} of the
geometry if $X \in \lieg_-$.

\item A \emph{distinguished chart} on $M$ is a chart with values
in $\lieg_-$ obtained as a local inverse of $\pi\o\exp_b|_{\lieg_-}$.  
\end{itemize}
\end{definition}


For a given parabolic model, exponential and distinguished curves can be classified by the \emph{geometric type} of the initial direction (see section
\ref{subsec.main.question} below).  Section 5.3 of \cite{cap.slovak.book.vol1} contains thorough descriptions of the
classes of distinguished curves for many parabolic models.  In the next section, we introduce \emph{normal charts}, with values in
tangent spaces.

\subsubsection{Adjoint tractor bundle and infinitesimal
  automorphisms}\label{1.1.2}

Natural vector bundles on parabolic geometries modeled on $G/P$ can be
obtained as associated bundles to the Cartan bundle.  Given a
representation $\Bbb W$ of $P$, form $B\x_P\Bbb W$. Using the Cartan
connection, such bundles can sometimes be identified with tensor
bundles; for example, the adjoint representation restricted to $P$
descends to the quotient vector space $\lieg / \liep$, and
$B\x_P(\lieg/\liep)\cong TM$. The main idea to define this isomorphism
is to note that for each $b\in B$ with $\pi(b)=x\in M$, the linear
isomorphism $\om_b:T_bB\to\lieg$ induces a linear isomorphism
$T_xM\cong T_bB/\ker(D_b\pi)\to\lieg/\liep$.

For $b\in B$ with $\pi(b)=x\in M$, the linear isomorphism
$T_xM\to\lieg/\liep\to\lieg_-$ composed with the normal coordinate
chart $\pi\o\exp_b$ gives a diffeomorphism from an open neighborhood
of $0$ in $T_xM$ onto an open neighborhood of $x\in M$, which we will
refer to as a \textit{normal coordinate chart} for $M$. By
construction, a straight line $\{t\xi:t\in \BR \}$ through $0 \in
T_xM$ corresponds to a distinguished curve through $x$ with initial
direction $\xi$.

Two further associated bundles to the Cartan bundle will play a role
in the sequel.  The Killing form of $\lieg$ gives an identification of
$(\lieg/\liep)^*$ with $\liep_+$, so $B\x_P\liep_+\cong T^*M$.  The
\textit{adjoint tractor bundle} is $\Cal AM=B\x_P\lieg$ and is useful
for studying infinitesimal automorphisms. The filtration of $\lieg$
gives rise to a filtration $\Cal AM=\Cal
A^{-k}M\supset\dots\supset\Cal A^kM$ by smooth subbundles. Since
$\lieg^1=\liep_+$ and $\lieg^0=\liep$, we see that $\Cal A^1M\cong
T^*M$ and $\Cal AM/\Cal A^0M\cong TM$.  Denote by $\Pi:\Cal AM\to TM$
the resulting natural projection.

The curvature of $\om$ from definition \ref{def:curvature} can be
naturally viewed as an element $\ka\in\Om^2(M,\Cal AM)$. Indeed, from
the defining properties of $\om$ it follows easily that
$K\in\Om^2(B,\lieg)$ is horizontal and $P$--equivariant and thus
corresponds to a form $\ka$ as above.

Vector fields on $B$ are in bijective correspondence with
$\lieg$--valued smooth functions via $\xi\mapsto \om(\xi)$.
Equivariance of $\om$ immediately implies that $\om(\xi)$ is a
$P$--equivariant function if and only if $\xi$ is a right--$P$--invariant vector field.  The
space $\Ga(\Cal AM)$ of smooth sections of $\Cal AM$ can be naturally
identified with the space $\frak X(B)^P$ of $P$--invariant vector
fields on $B$; note that these descend to $M$.
On the bundle $\mathcal{A}M$, the corresponding projection is $\Pi$.

An \emph{automorphism} of $(M,B,\omega)$ is a principal bundle
automorphism that preserves $\omega$.  These form a Lie group, which
will be denoted $\mbox{Aut}(M,B,\omega)$.  An \emph{infinitesimal
  automorphism} is given by $\xi\in\frak X(B)^P$ such that $\Cal
L_\xi\om=0$, where $\Cal L$ denotes the Lie derivative. (For a pseudo-Riemannian metric, infinitesimal automorphisms are called Killing fields.) 
An infinitesimal automorphism $\tilde{\eta}$ descends
to a vector field $\eta$ on $M$.  The resulting subalgebra of
$\frak{X}(M)$ will be denoted $\alginf(M)$ below; note these vector
fields are not assumed to be complete.

\subsubsection{Normal parabolic geometries and harmonic curvature}

Parabolic geometries encode certain underlying geometric structures.
First note that via the isomorphism $TM\cong \Cal AM/\Cal A^0M$, a
parabolic geometry of type $(\lieg,P)$ gives rise to a filtration
$TM=T^{-k}M\supset\dots\supset T^{-1}M$ of the tangent bundle, where
$T^iM=\Cal A^iM/\Cal A^0M$.  This filtration gives rise to a
filtration of $\Omega^2(M,\Cal AM)$ by homogeneity:
$\tau\in\Om^2(M,\Cal AM)$ is called \emph{homogeneous of degree
  $\geq\ell$} if for $\xi\in T^iM$ and $\eta\in T^jM$ the value
$\tau(\xi,\eta)\in \Cal A^{i+j+\ell}M$.

The geometry $(M,B,\om)$ is \textit{regular} if the curvature 2-form
$\kappa$ is homogeneous of degree at least 1.  A Cartan geometry is
\emph{torsion-free} if $\ka$ has values in $\Cal A^0M\subset\Cal AM$;
torsion-free implies regular.  Now assuming regularity, the underlying
structure of a parabolic geometry consists of the filtration
$\{T^iM\}_{i \in \BN}$ of the tangent bundle and a reduction of structure group of
the associated graded of this filtered bundle to 
$G_0$.  Conversely, any such structure of a filtration with a $G_0$-reduction can be obtained from
some regular parabolic geometry.

These geometric structures are equivalent, in the categorical sense,
to regular parabolic geometries satisfying an additional condition on
$\kappa$ called \emph{normality}.  The Lie algebra homology
differentials for the Lie algebra $\frak p_+$ with coefficients in the
module $\lieg$ defines a $P$--equivariant homomorphism $\La^k\frak
p_+\otimes\lieg\to \La^{k-1}\frak p_+\otimes\lieg$, which is
traditionally denoted by $\partial^*$ and called the \textit{Kostant
  codifferential}.  For $k=2$, this homomorphism gives on the level of
associated bundles a natural bundle map $\La^2T^*M\otimes\Cal AM\to
T^*M\otimes\Cal AM$, also denoted by $\partial^*$. Now the geometry is
called \textit{normal} if $\partial^*\ka=0$.

The projective and conformal structures mentioned above correspond to
$|1|$--gradings.  Other geometric structures arising from parabolic
Cartan geometries include almost-Grassmannian and almost-quaternionic
structures, hypersurface--type CR structures, path geometries, and
several types of generic distributions.

The equivalence in the categorical sense implies that any automorphism
of the underlying structure uniquely lifts to an automorphism of the
parabolic geometry. The analogous result for vector fields says any
infinitesimal automorphism $\eta$ of the underlying structure lifts
uniquely to $\tilde\eta\in\frak X(B)^P$ such that $\Cal
L_{\tilde\eta}\om=0$. Conversely, projecting an infinitesimal
automorphism of $(M,B,\om)$ to $M$ gives an infinitesimal automorphism
of the underlying structure.

The normality condition for parabolic geometries can
also be used to extract the essential part of the curvature of the
canonical Cartan connection. As mentioned above, the Kostant
codifferential induces bundle maps
$$
\La^3T^*M\otimes\Cal AM\to\La^2T^*M\otimes\Cal AM\to T^*M\otimes\Cal AM.
$$ 
Since these maps come from a homology differential, the composition
of the two bundle maps above is zero, so there are natural subbundles
$\im(\partial^*)\subset\ker(\partial^*)\subset \La^2T^*M\otimes\Cal
AM$. By construction, the quotient bunde $\Cal
H_2=\ker(\partial^*)/\im(\partial^*)$ can be realized as
$B\x_PH_2(\frak p_+,\frak g)$, and the latter Lie algebra homology
group can can identified with the Lie algebra cohomology group
$H^2(\frak g_-,\lieg)$. For a regular normal parabolic geometry, the
curvature $\ka$ actually is a section of $\ker(\partial^*)$.  The
\textit{harmonic curvature} $\ka_H\in\Ga(\Cal H_2)$ is the image of
$\ka$ under the obvious quotient projection.

General theorems assert that no information is lost in passing from
$\ka$ to $\ka_H$. First, vanishing of $\ka_H$ on an open subset
$U\subset M$ implies vanishing of $\ka$ on $U$. In fact, there is a
natural differential operator $S:\Ga(\Cal H_2)\to\Om^2(M,\Cal AM)$
such that $S(\ka_H)=\ka$. The crucial advantage of the harmonic
curvature is that one can show that the representation $H_2(\frak
p_+,\frak g)$ is always completely reducible, so the corresponding
associated bundle is a simpler geometric object than
$\La^2T^*M\otimes\Cal AM$.  The structure of $H_2(\frak p_+,\frak g)$
can be computed with Kostant's version of the Bott--Borel--Weil
theorem---see section 3.3 of \cite{cap.slovak.book.vol1}. We will use
without further citation the resulting descriptions of harmonic
curvature components for the individual geometries we discuss.

\subsection{Higher order fixed points and the main questions}
\label{subsec.main.question}

Let $(M,B,\omega)$ be a regular normal parabolic geometry of type
$(\lieg,P)$.  Let $\eta \in \alginf(M)$, and denote the induced vector
field on $B$ by $\tilde\eta$ and by $s$ the corresponding section of
$\Cal AM$.  For $x_0 \in M$, we have $\eta(x_0) = 0$ if and only if $s(x_0)\in
\Cal A^0M$, equivalently if $\om_{b_0}(\tilde\eta)\in\frak p$ for
any $b_0 \in \pi^{-1}(x_0)$.

\begin{definition}
The infinitesimal automorphism $\eta$, or the corresponding section
$s$ of $\mathcal{A}M$, has a \emph{higher order fixed point} at
$x_0$ if $s(x_0)\in\Cal A^1M$.  In this case, the \textit{isotropy at
  $x_0$} of $\eta$ (or of $s$) is the element of the cotangent space
corresponding to $s(x_0)$.
\end{definition}


Via $b_0\in\pi^{-1}(x_0)$ the isotropy $\al\in T^*_{x_0}M$ corresponds
to an element in $\frak p_+$. The \textit{geometric type} of the
infinitesimal automorphism at $x_0$ is defined to be the $P$--orbit of
this element, which is independent of the choice of $b_0$.  These
geometric types give rise to an initial classification of higher-order
fixed points.  For example, in the conformal case, there is a natural
inner product on $T^*_{x_0}M$ up to scale, so isotropies can be postive,
null, or negative.  For algebraically more complicated models, as for
CR structures, the cotangent bundle has a natural filtration induced
by $\Cal A^1M\supset\dots\supset\Cal A^kM$, which leads to a variety
of possible geometric types of isotropies.  These will be discussed in detail in several
examples below. In all cases we consider there are only finitely many
orbits, and in general there are finiteness results of E.~Vinberg \cite{vinberg.graded}. 

{\bf Main Questions}: \emph{What special dynamical properties are shared by
infinitesimal automorphisms admitting a higher order fixed point with
isotropy of a certain geometric type?  What curvature restrictions are implied by existence of such an automorphism?  Which types of higher order
fixed points in $x_0$ imply that $(M,B,\omega)$ is locally flat on a
nonempty open set $U$ with $x_0 \in \overline{U}$?}



\begin{remark}
  The concept of essential automorphisms, which previously existed for
  conformal and projective structures, has recently been extended to
  all parabolic geometries by J.~Alt in \cite{alt.essential.def}.  It
  is immediate from his proposition 3.4 that a Killing field with a
  higher order fixed point is essential.
\end{remark}

\subsection{Results}

We develop general tools in section \ref{sec.general} that give the
precise action of a flow on specific curves emanating from a higher
order fixed point.  Several of these propositions generalize the tools
of \cite{nagano.ochiai.proj} and \cite{fm.nilpconf}.  We can apply
them to recover the previously cited theorems on higher order fixed
points for projective and conformal flows.  These applications are
presented in a separate article \cite{cap.me.proj.conf}.

In section \ref{sec.grassmann}, we apply these tools to
$(2,n)$-almost-Grassmannian structures to describe the two types of
higher order fixed points in these geometries, and we show in theorem
\ref{thm.grassmannian} that if the geometry is torsion-free, then
existence of either type implies flatness on a nonempty open set.  The
proofs in this section easily adapt to prove an analogous result for
almost-quaternionic structures, without any torsion-freeness assumption
(see theorem \ref{thm.quaternion}).

Next we prove a general result for parabolic contact structures (see
section \ref{sect:contact}): another natural generalization of the
hypotheses in the projective and conformal results
\cite{nagano.ochiai.proj} and \cite{fm.champsconfs} is to assume that
a flow fixes a point $x_0$ and has trivial derivative at $x_0$, which
in the parabolic contact case is a stronger assumption than $x_0$
being a higher order fixed point.  Under this hypothesis, we prove in
theorem \ref{thm.contact} that the curvature vanishes on an open set
with $x_0$ in its closure; in some cases, we can further deduce
flatness on a neighborhood of $x_0$.

Finally, in section \ref{sect:CR}, we treat partially integrable
almost-CR structures.  Theorem \ref{thm.cr} says the harmonic
curvature always vanishes at a higher order fixed point, and certain
types of higher order fixed points imply flatness on a nonempty open
set.  A consequence of this theorem is a local $C^\omega$ version of
the Schoen-Webster theorem on automorphisms of strictly pseudoconvex
CR structures (see theorem \ref{thm.strict.cr}; compare also
\cite{vitushkin.local.schoen} and \cite{kruzhilin.local.schoen}.).
For general nondegenerate CR structures, we are left with the
following question. (Beloshapka \cite{beloshapka.str.ess.cr} and Loboda
\cite{loboda.str.ess.cr} have proved the answer is no in the case of
real-analytic hypersurfaces of complex manifolds.)

\begin{question}
Is there a non-flat partially integrable almost-CR manifold, the
automorphism group of which has a higher-order fixed point?
\end{question}

\section{General results}
\label{sec.general}

We note that the holonomy calculations of propositions
\ref{prop.attractor} and \ref{prop.sl2.triple} below generalize lemma
5.5 of \cite{nagano.ochiai.proj} from the projective setting and
proposition 4.5 of \cite{fm.nilpconf} from the conformal case.

We begin with a basic proposition from \cite{fm.nilpconf} that
computes the holonomy of an automorphism of a Cartan geometry with
isotropy $g$ in terms of the action of $g$ on $G$. 

\begin{proposition} \cite[prop 4.3]{fm.nilpconf}
\label{prop.basic.holonomy}
Let $\varphi \in \mbox{Aut}(M,B,\omega)$.
Suppose
\begin{itemize}
\item $\varphi b_0 = b_0 g$ for some $b_0 \in B$ and $g=g_0 \in P$
\item $\exp(b_0,sU)$ is defined for $s$ in an interval $I$ around $0$
  and for some $U \in \lieg$
\item $g e^{sU} = e^{c(s) U} p(s)$ in $G$, where $p(t) : I \rightarrow
  P$ with $p(0) = g_0$, and $c : I \rightarrow I'$ is a diffeomorphism
  fixing $0$.
\end{itemize}
Then the corresponding equation holds in $B$: $\exp(b_0,sU)$ is also
defined on $I'$, and
$$ \varphi \exp(b_0,sU) = \exp(b_0,c(s)U) p(s)$$
 \end{proposition}

\begin{remark}
The proposition says that in normal coordinates for
$(M,B,\omega)$ centered at the fixed point $\pi(b_0)$, the
automorphism $\varphi$ resembles the model automorphism $g$ acting on
$G/P$ with fixed point $o = [P]$. The proposition is related to the \emph{comparison maps} studied in the recent
paper \cite{cgh.holonomy}.
\end{remark}

Before stating our first general results on higher order fixed points, we introduce some terminology.

\begin{definition}  Let $\eta \in \alginf(M)$.
\begin{itemize} 
\item The \emph{strongly fixed set} of a given geometric type of the flow generated by $\eta$ is the set of all
higher order fixed points of that type. For a higher order fixed point $x_0$, the term strongly fixed
set will mean the strongly fixed set of the type of $x_0$. 

\item Given a neighborhood $U$ of a higher order fixed point $x_0$ the
\emph{strongly fixed component} of $x_0$ in $U$ is the set of all
points that can be reached from $x_0$ by a smooth curve contained in
the intersection of $U$ with the strongly fixed set of $x_0$. The
higher order fixed point $x_0$ is called \emph{smoothly isolated}
if $\{ x_0 \} $ equals the strongly fixed component in some neighborhood.
\end{itemize}
\end{definition}

For any $x_0\in M$ and any choice of $b\in\pi^{-1}(x_0)$, the Cartan
connection $\omega$ gives identifications $T^*_{x_0}M\cong \frak p_+$
and $T_{x_0}M\cong\lieg/\liep$ such that the duality between the two
spaces is induced by the Killing form of $\lieg$.  Then any $\alpha\in
T^*_{x_0}M$ corresponds to an element $Z\in\frak p_+$. Put
\begin{gather*}
F_\lieg(Z)=\{X\in\lieg:\ad_X^k(Z)\in\liep\quad\forall k\in \BN \}\\ 
C_\lieg(Z)=\{X\in\lieg : [X,Z]=0\}\subset F_\lieg(Z)
\end{gather*}
A different choice $b'\in \pi^{-1}(x_0)$ is of the form $b'=b\cdot g$
for some $g \in P$, which leads to $Z'=\Ad(g^{-1})(Z)$;
 $F_\lieg(Z')=\Ad(g^{-1})(F_\lieg(Z))$;  $C_\lieg(Z')=\Ad(g^{-1})(C_\lieg(Z))$. By
point 2(b) in definition \ref{def:cartan.geometry} the images of these
subsets in $\lieg/\liep$ determine subsets $C(\alpha)\subset
F(\alpha)\subset T_{x_0}M$ which are independent of the choice of
$b$. Of course, $C(\alpha) \subset T_{x_0}M$ is a linear subspace.

\begin{definition}\label{def:commutant}
The subset $F(\al)\subset T_{x_0}M$ and the subspace $C(\al)\subset
T_{x_0}M$ determined by $\al\in T^*_{x_0}M$ as above are called,
respectively, the \emph{normalizing subset} and the \emph{commutant}
of $\al$ in $T_{x_0}M$.
\end{definition}

The following is a consequence of proposition
\ref{prop.basic.holonomy} for the local behavior of an infinitesimal
automorphism around a higher order fixed point.

\begin{proposition}\label{prop.attractor}
  Consider a Cartan geometry $(M,B,\om)$ modeled on $(\lieg,P)$ and
  $\eta\in \alginf(M)$ with higher order fixed point $x_0\in M$
  with isotropy $\al\in T^*_{x_0}(M)$. 

\begin{enumerate}
\item For any $\xi\in F(\al)\subset T_{x_0}M$, there is an exponential
  curve emanating from $x_0$ in the direction $\xi$ consisting of
  fixed points for $\eta$.  If $\xi\in C(\al)\subset F(\al)$, then
  this curve lies in the strongly fixed set of $x_0$.
  
\item If $k = \dim C(\al)$, then there is a $k$--dimensional
  submanifold $N\subset M$ through $x_0$ contained in the strongly
  fixed set of $x_0$, with $T_{x_0}N=C(\al)$.
  
\item For some neighborhood $V$ of $x_0$ any point in the strongly
  fixed component of $x_0$ in $V$ can be reached by an exponential
  curve emanating from $x_0$ in a direction belonging to
  $C(\alpha)$. In particular, if $C(\al)=\{0\}$, then $x_0$ is
  smoothly isolated. 
\end{enumerate}
\end{proposition}

\begin{Pf}
Choose $b_0\in\pi^{-1}(x_0)$ and set $Z=\om_{b_0}(\tilde\eta)\in\frak
p_+$. Let $U$ be a neighborhood of $0$ in $\lieg$ on which the
restriction of $\exp_{b_0}$ is defined and is a diffeomorphism onto
its image.  For $\xi\in F(\al)$, there is an element $X\in F_\lieg(Z)$
such that $D_{b_0}\pi (\om^{-1}(X))=\xi$.  Let $\gamma(r) =
\pi(\exp_{b_0}(rX))$, defined for small $r$.

In $G$ we have $e^{tZ}e^{rX}=e^{rX}e^{t\Ad(e^{-rX})(Z)}$ and 
$$
\Ad(e^{-rX})(Z)=\sum_{k=0}^\infty\frac{(-r)^k}{k!}\ad_X^k(Z),
$$ 
which lies in $\frak p$ since $X\in F_\lieg(Z)$.  If $\xi\in
C(\al)$, then $X$ can be chosen in $C_\lieg(Z)$, and then 
$\Ad(e^{-rX})(Z)=Z$.  
It follows from proposition \ref{prop.basic.holonomy} that in $B$, for
all sufficiently small $t$, $X\in V$, and $r\leq 1$,
$$ 
\varphi^t_{\tilde{\eta}} \exp(b_0,rX) = \exp(b_0,rX)e^{t\Ad(e^{-rX})(Z)}
$$ 

Now (1) follows for $F(\alpha)$ because the rightmost term is in $P$. For $\xi \in C(\alpha)$, differentiate
with respect to $t$ at time 0 to see that $\tilde\eta(\exp(b_0,rX))$
coincides with the fundamental vector field generated by $Z$ for all $r$.

For point (2), note that we can choose a $k$--dimensional
subspace $C\subset C_\lieg(Z)$ with $C \cap \frak p = 0$. Extend $C$
to a linear subspace of $\lieg$ which is complementary to
$\liep$. Then $\pi\o\exp_{b_0}$ can be restricted to an open
neighborhood of zero to obtain a submanifold chart as required.

For (3), let $\gamma(r)$ be a smooth curve emanating from $x_0$ which lies
in the strongly fixed set of $x_0$. Then there is a lift
$\tilde{\beta}$ of $\gamma$, which we may assume begins at $b_0$,
satisfying
$$ \tilde{\eta}_{\tilde{\beta}(r)} = (\tilde{Y}_r)_{\tilde{\beta}(r)},$$
 
for some $Y_r \in \liep$ with corresponding fundamental vector field
$\tilde{Y}_r$; moreover, each $Y_r$ is conjugate in $P$ to $Z$, and
$Y_0 = Z$.  Now the $P$--conjugacy class of $Z$ is in bijection with $P /
C_P(Z)$, where $C_P(Z)$ is the stabilizer in $P$ of $Z$ under the
adjoint action.  A smooth path in this quotient can be lifted to
$P$. Hence on a sufficiently small interval, we obtain a smooth path
$c_r$ through the identity in $P$ such that $\Ad(c_r)(Y_r) = Z$.
Then let $\tilde{\gamma} = \tilde{\beta} c_r^{-1}$.  For this lift,

$$ \tilde{\eta}_{\tilde{\gamma}(r)} = (\widetilde{\Ad(c_r)(Y_r}))_{\tilde{\beta}(r)} = \tilde{Z}_{\tilde{\beta}(r)}$$

Therefore $\varphi^t_{\tilde{\eta}} \tilde{\gamma}(r) =
\tilde{\gamma}(r) e^{tZ}$ for sufficiently small $r$. Fix $r$ for
which $\tilde{\gamma}(r) \in \exp_{b_0}(U)$, so there is $X \in U$
such that $\exp_{b_0}(X) = \tilde{\gamma}(r)$.  Then
$$ \varphi^t_{\tilde{\eta}} \exp_{b_0}(X) = \exp_{b_0}(X) e^{tZ} $$
for all $t$.  On the other hand, the expression above also equals
$$ \exp_{b_0 e^{tZ}} (X) = \exp_{b_0}(\Ad(e^{tZ})(X)) e^{tZ}$$

Therefore $\exp_{b_0}(X) = \exp_{b_0}(\Ad(e^{tZ})(X))$, so for any
$t$ sufficiently small that $\Ad(e^{tZ})(X) \in U$, we have $\Ad(e^{tZ})(X) = X$, and thus $[Z,X] = 0$, which completes the proof of
(3).
\end{Pf}

\begin{remark} 
In most cases we discuss, we can strengthen this result
by showing that the curves in (1) are distinguished curves for the
geometry and describing the submanifold in (2) in terms of normal
coordinates. 
These improvements will be presented
at the end of section \ref{sec.general}.
\end{remark}

To proceed further, we need an analog of the concept of holonomy
sequences associated to sequences of automorphisms of a Cartan
geometry, which appeared in \cite{frances.lfrank1}.  The notion was
further developed in later papers, including \cite{bfm.zimemb},
\cite{fm.nilpconf}, \cite{frances.riemannien}, and
\cite{frances.degenerescence}.  The following definition for flows is
the most useful variation for our purposes.

\begin{definition}\label{def.holonomy}
  Let $\{ \varphi^t \}$ be a flow by automorphisms of $(M,B,\omega)$, and
  let $b \in B$.  A path $p(t) \in P$ is a \emph{holonomy path at $b$
    with attractor $b_0 \in B$} for $\{ \varphi^t \}$ if there exists a path
  $b(t)$ and a point $b_0$ in $B$ with
$$ b(t) \rightarrow b \qquad \mbox{and} \qquad \varphi^t b(t)
 p(t)^{-1} \rightarrow b_0$$ as $t \rightarrow \infty$.
\end{definition}

The following proposition provides a condition under which a holonomy
path at one point can be propagated to nearby points.

\begin{proposition}
\label{prop.ext.holonomy}
Let $\{ \varphi^t \}$ be a flow by automorphisms of $(M,B,\omega)$.  Let
$p(t)$ be a holonomy path at $b$ with attractor $b_0$ and path $b(t)$ as in definition \ref{def.holonomy}.  Suppose that
for some $Y\in\lieg$, $\exp(b(t),Y)$ is defined for all $t$ and $\Ad
(p(t))(Y) \rightarrow Y_\infty$ as $t \rightarrow \infty$.  Then $p(t)$
is a holonomy path at $\exp(b,Y)$ with attractor $\exp(b_0,Y_\infty)$.
\end{proposition}

\begin{Pf}
We have
$$\varphi^t b(t) p(t)^{-1} =k(t) \rightarrow b_0$$

Then $\exp(b(t),Y) \rightarrow \exp(b,Y)$, and
\begin{eqnarray*}
  \varphi^t \exp(b(t),Y) p(t)^{-1} & = & \exp(\varphi^t b(t),Y) p(t)^{-1} \\
  & = & \exp(k(t) p(t),Y) p(t)^{-1} \\
  & = & \exp(k(t), \Ad(p(t))Y)  \rightarrow \exp(b_0,Y_\infty)
\end{eqnarray*}
\end{Pf}

A holonomy path $p(t)$ at $b$ leads to restrictions on the possible
$\ph^t$--invariant sections of any bundle associated to the Cartan
bundle.  Given a representation $\Bbb W$ of $P$, a principal bundle
automorphism $\ph$ of $B$ gives rise to an automorphism of the
associated bundle $B\x_P\Bbb W$. If $\ph$ also preserves $\om$, and the
associated bundle is a tensor bundle, then this automorphism is the
one functorially associated to $\ph$.
Smooth sections of
$B\x_P\Bbb W$ correspond to smooth $P$-equivariant maps $f:B\to\Bbb W$---that is, $f(b g)=g^{-1} f(b)$ for
any $g\in P$. 
The pullback of a section by $\ph$ corresponds to the precomposition
of $f:B\to\Bbb W$ with $\ph$. In particular, if $f$ corresponds to a
$\ph^t$--invariant section for a flow, then $f\o\ph^t=f$ for all $t$.

The Cartan curvature $\ka$, and components of the harmonic curvature
for parabolic geometries, are both invariant under automorphisms
because the corresponding sections are constructed naturally from the
Cartan connection. For the Cartan curvature, the representation $\Bbb
W$ is $\La^2(\frak g/\frak p)^*\otimes\frak g$, which in general is
rather complicated; for the components of the harmonic curvature, in
contrast, the representation is always irreducible.  The number of harmonic curvature components and the form of the corresponding representations
varies according to the type of geometry in question.

\begin{proposition}
\label{prop.curv.stab}
Let $p(t)$ be a holonomy path for $\{ \varphi^t \}$ corresponding to
$b(t)\to b$ with attractor $b_0$, and let $f:B\to\Bbb W$ be the
equivariant function corresponding to a $\ph^t$--invariant section of
$B\x_P\Bbb W$. Put $k(t)=\varphi^t(b(t)) p(t)^{-1}$ as in
definition \ref{def.holonomy}.
\begin{enumerate}
\item Then $p(t)\cdot f(b(t))\to f(b_0)$ as $t\to\infty$.

\item Assume moreover that $p(t)$ is contained in a $1$--parameter subgroup of
$P$ that is diagonalizable on $\Bbb W$, and let $\Bbb W = \Bbb W_0
\oplus \cdots \oplus \Bbb W_\ell$ be an eigenspace decomposition with
eigenvalues given by functions $\lambda_i(t)$, $i = 1, \ldots, \ell$.  Let $f(b)_i$ be
the component of $f(b)$ in $\Bbb W_i$.
\begin{itemize}
\item If $\lambda_i(t) \rightarrow \pm \infty$ as $t \rightarrow \infty$,
  then $f(b)_i = 0$.
\item 
If $f(b)_i \neq 0$, then $||f(k(t))_i || \in
\Theta(\lambda_i(t))$, where $|| \cdot ||$ is any norm on $\Bbb{W}$;
in particular, if $f(b_0)_i=0$ but $\lambda_i(t)$ does not tend to $0$
as $t\to\infty$, then $f(b)_i = 0$.
\end{itemize}
\end{enumerate}
\end{proposition}

Recall that for functions $f$ and $g$ defined on $\BR^+$, the notation $f(t) \in \Theta(g(t))$ means there exist nonzero constants $c, C \in \BR^+$ such that 
$$cg(t) \leq f(t) \leq C g(t) \qquad \mbox{for all} \qquad t \geq 0$$

\begin{Pf}
 
  For (1), we see that the $P$-equivariance of $f$ implies
  $$f(k(t))=p(t)\cdot f(\ph^t(b(t))),$$ and $\ph^t$-invariance implies
  $f(\ph^t(b(t)))=f(b(t))$; now (1) follows.

For (2) we have by (1) that
\begin{eqnarray*}
\lambda_i(t)f(b(t))_i\to f(b_0)_i \quad \text{ for } t\to\infty  
\end{eqnarray*}
Since $f(b(t))\to f(b)$ for $t\to\infty$, the first property follows
immediately.  For the second, also use that $f(k(t))_i = \lambda_i(t) f(b(t))_i$.  
%
\end{Pf}


\begin{definition}
\label{def.p.stable.subspace}
Let $p(t)$ be a holonomy path at $b$ and $\Bbb W$ a $P$-representation
satisfying the hypotheses of part (2) of proposition
\ref{prop.curv.stab} above.  Let $\Bbb W_i$ and $\lambda_i(t)$, $i =
0, \ldots, l$, be as above.
\begin{itemize}
\item The \emph{stable subspace} for $p$, denoted $\Bbb W_{st}(p)$, is
  the sum of the eigenspaces $\Bbb W_i$ for which $\lambda_i(t)$ is
  bounded as $t \rightarrow \infty$.

\item The \emph{strongly stable subspace} for $p$, denoted $\Bbb
  W_{ss}(p)$, is the sum of eigenspaces $\Bbb W_i$ with $\lambda_i(t)
  \rightarrow 0$ as $t \rightarrow \infty$.
\end{itemize}
\end{definition}

Part (2) of proposition \ref{prop.curv.stab} says that we must
always have $f(b)\in\Bbb W_{st}(p)$, and if $f(b_0)=0$, then also
$f(b)\in\Bbb W_{ss}(p)$.

A crucial fact for the sequel is that in the case of parabolic
geometries, nice holonomy paths can be obtained from purely algebraic
data.  Recall that an \emph{$\frak{sl}_2$--triple} in a Lie algebra
$\frak a$ is formed by elements $E,H,F\in\frak a$ such that $[E,F]=H$,
$[H,E]=2E$, and $[H,F]=-2F$.  The Jacobson--Morozov theorem says that any element $E$ in the semisimple Lie algebra
$\frak g$ which is nilpotent in the adjoint representation can be
completed to an $\frak{sl}_2$--triple: there exists $F \in \lieg$ such
that $E$, $[E,F]$, and $F$ form an $\frak{sl}_2$--triple.

Let $x_0\in M$ and $\al\in T^*_{x_0}M$.  As above, a choice of $b_0
\in\pi^{-1}(x_0)$ associates to $\al$ an element $Z\in\frak p_+$,
which is nilpotent. We define $T_\lieg(Z)$ to be the non--empty set of
all $X\in\lieg$, such that $Z$, $[Z,X]$, and $X$ form an
$\frak{sl}_2$--triple. As for the normalizing set and the commutant,
we obtain a subset $T(\al)\subset T_{x_0}M$ which is independent of
the choice of $b$. 

\begin{definition}
For $\alpha \in T_{x_0}^*M$, the non--empty subset $T(\al)\subset
T_{x_0}M$ defined above is called the \emph{counterpart set} of $\alpha$.

\end{definition}



Now we can precisely compute the action of $\varphi^t$ on certain curves from
$x_0$ in any direction belonging to the counterpart set of the
isotropy.
 
\begin{proposition}\label{prop.sl2.triple}
Let $(M,B,\om)$ be modeled on $(\lieg,P)$, and let $\eta\in\alginf(M)$
have a higher order fixed point at $x_0\in M$ with isotropy $\al\in
T^*_{x_0}M$.  Let $\xi\in T(\al)$, and let $Z,X\in\lieg$ be the elements associated
to $\al$ and $\xi$, respectively, for a choice of
$b_0\in\pi^{-1}(x_0)$; set $A=[Z,X]$.

Then there exists an exponential curve $\sigma:(- \epsilon,
\epsilon) \to M$, for some $\epsilon > 0$, such that $\sigma(0)=x_0$,
$\sigma'(0)=\xi$, and
$$ \ph^t_\eta(\sigma(s)) =\sigma\left( \frac{s}{1+st} \right)
\ \mbox{whenever\ }  |s| < \epsilon \ \mbox{and} \ ts \geq 0$$

Moreover, for each such $s$, there is $b\in\pi^{-1}(\sigma(s))$
such that $g(t)=e^{\log(1+st) A}$ is a holonomy path at $b$ with
attractor $b_0$. 
\end{proposition}

\begin{Pf}
  In $SL(2,\BR)$, compute
$$
\begin{pmatrix} 1 & t \\ 0 & 1\end{pmatrix}\begin{pmatrix}1 & 0 \\
s & 1\end{pmatrix}=\begin{pmatrix}1 & 0 \\ \frac{s}{1+st} &
1\end{pmatrix}\begin{pmatrix}1 + st & 0 \\ 0 &
\frac{1}{1+st}\end{pmatrix}\begin{pmatrix}1 & \frac{t}{1+st} \\
0 & 1 \end{pmatrix}.
$$
The $\frak{sl}_2$--triple in $\frak g$ formed by $Z$, $A$, and
$X$ gives rise to a Lie algebra homomorphism $\frak{sl}(2,\BR)\to\lieg$, which locally integrates to a group homomorphism
$SL(2,\BR)\to G$.  The above equation then shows that in $G$ 
$$
e^{tZ}e^{sX}=e^{\frac{s}{1+st} X}e^{\log(1+st) A}e^{\frac{t}{1+st} Z}.
$$
Proposition \ref{prop.basic.holonomy} shows that the local flow
$\ph^t_{\tilde{\eta}}$ of $\tilde\eta$ satisfies
\begin{equation}
  \label{eq:sl2-hol-path}
\ph^t_{\tilde{\eta}}\exp(b_0,sX)=\exp(b_0,\frac{s}{1+st} X)e^{\log(1+st)
  A}e^{\frac{t}{1+st} Z}  
\end{equation}
for sufficiently small $s$. Fixing such a value for $s$, the
path $b(t)=\exp(b_0,sX)e^{\frac{-t}{1+st} Z}$ evidently satsfies
$b(t)\to b=\exp(b_0,sX)e^{-\frac1s Z}$, and
$\pi(b)=\pi(\exp(b_0,sX))$.  Define $\sigma(s)$ to be this latter curve, valid on some nonzero interval $(- \epsilon, \epsilon)$.


Note $\sigma'(0) = D_{b_0}\pi( \omega_{b_0}^{-1}(X)) = \xi$, and $\sigma'(s) \neq 0$ for all $|s| < \epsilon$. Finally, equation
\eqref{eq:sl2-hol-path} says that $\ph^t_{\tilde{\eta}}(b(t))\cdot e^{-\log(1+st)
  A}=\exp(b_0,\frac{s}{1+st} X)$, which converges to
$b_0$ for $ts > 0$ as $|t| \to\infty$.  
\end{Pf}


Basic representation theory says that in a standard basis of
$\frak{sl}(2,\BR)$, the semisimple element $H$ acts diagonalizably
with integer eigenvalues in any finite-dimensional complex
representation.  Thus for an $\frak{sl}_2$--triple $Z$, $A$, $X$
in $\lieg$, the endomorphism $\ad(A)$ is diagonalizable on $\lieg$, so
$A\in\lieg$ is a semisimple element.  Further, $A$ acts diagonalizably
with integer eigenvalues on any finite-dimensional complex
representation of $\lieg$.  We will assume all representations are
finite-dimensional below.

If we assume in addition that $A\in\lieg_0$, then $\ad(A)(\lieg_0)\subset\lieg_0$, so $\ad(A)$ acts diagonalizably on
$\lieg_0$. Now $\lieg_0$ is reductive, so it is the
direct sum of the center plus a semisimple subalgebra. The component
of $A$ in the center acts trivially under $\mbox{ad}$, so the
semisimple component of $A$ acts diagonalizably under $\mbox{ad}$, and hence
in any finite-dimensional representation. Therefore, $A$ acts
diagonalizably in any finite-dimensional representation of $\lieg_0$
in which the center acts diagonalizably. All representations
corresponding to the components of the harmonic curvature have this
property because they are subquotients of representations of $\lieg$
in which the center of $\lieg_0$ is contained in a Cartan subalgebra.  (All real representations of interest we are aware of have
this property.)


\begin{definition}\label{def-gradingcomp}
  Let $A\in\lieg_0$, and let $\Bbb W$ be a representation of $\lieg_0$
  on which $A$ acts diagonalizably with eigenspace decomposition
  $\Bbb{W} = \Bbb{W}_0 \oplus \cdots \oplus \Bbb{W}_l$ and eigenvalues
  $\mu_i$, $i = 0, \ldots, l$.
  \begin{itemize}
\item The \emph{stable subspace} for $A$, denoted $\Bbb W_{st}(A)$, is
  the sum of the eigenspaces $\Bbb W_i$ for which $\mu_i \leq 0$.

\item The \emph{strongly stable subspace} for $A$, denoted $\Bbb
  W_{ss}(A)$, is the sum of eigenspaces $\Bbb W_i$ with $\mu_i < 0$.
\end{itemize}
\end{definition}

Now our results have the following useful formulation:

\begin{corollary}\label{cor:sl2-path}
In the setting of proposition \ref{prop.sl2.triple}, assume that
$A=[Z,X]\in\lieg_0$ and let $\Bbb W$ be a representation of
$\lieg_0$.   Let $I = (- \epsilon, \epsilon)$.
\begin{enumerate}
\item If $\Bbb W_{st}=\{0\}$, then any $\ph^t_\eta$--invariant section of
$B\x_P\Bbb W$ vanishes on $\sigma(I)$.

\item If $\Bbb W_{ss}=\{0\}$, then any $\ph^t_\eta$--invariant section of
$B\x_P\Bbb W$ which vanishes at $x_0$ vanishes on all of $\sigma(I)$.

\item Suppose that $\Bbb W_{ss}=\{0\}$ and that all eigenvalues of $A$
on $\lieg_-$ are non--positive with the $0$-eigenspace equal
to $C_{\lieg}(Z) \cap \lieg_-$. Then any $\ph^t_\eta$--invariant section of $B\x_P\Bbb W$
that vanishes at each fixed point of the same geometric type as $x_0$ in a neighborhood 
also vanishes on an open neighborhood of $\sigma(I \setminus \{ 0 \} )$ and thus on
an open subset containing $x_0$ in its closure.
\end{enumerate}
\end{corollary}

\begin{Pf}
  The $\mu_i$--eigenspace $\Bbb W_i$ for $A$ is the eigenspace for
  $g(t)=e^{\log(1+st) A}$ with eigenvalue $\lambda_i(t) =
  (1+st)^{\mu_i}$. Thus $\Bbb W_{st}(A)=\Bbb W_{st}(g)$ and $\Bbb
  W_{ss}(A)=\Bbb W_{ss}(g)$ for the holonomy path $g(t)$. After definition \ref{def.p.stable.subspace},
  we observed that if $f:B\to\Bbb W$ is the function
  corresponding to an invariant section of the bundle $B\x_P\Bbb W$,
  then $f(b)\in \Bbb W_{st}(A)$, and if $f(b_0)=0$, then $f(b)\in\Bbb
  W_{ss}(A)$. These facts together with proposition \ref{prop.sl2.triple} yield (1) and (2). 

  The additional assumption in (3) implies that for any
  $Y\in\lieg_-$, the limit of $\Ad(g(t))(Y)$ as
  $t\to\infty$ exists, and equals some $Y_\infty \in C_{\lieg}(Z) \cap \lieg_-$. Taking a point $b\in
  B$ over a point $\sigma(s) \neq x_0$ close enough to $x_0$,
  we can thus invoke proposition \ref{prop.ext.holonomy} to say that $g(t)$ is a holonomy path at $\exp(b,Y)$ with attractor $\exp(b_0,Y_\infty)$ (assuming $Y$ sufficiently close to $0$). But $\exp(b_0,Y_\infty)$ lies over a
  higher order fixed point of the same geometric type as $x_0$, by
  proposition \ref{prop.attractor} part (1).  Then any $\varphi^t_\eta$-invariant section must vanish at $\exp(b,Y)$, as well.  Varying $Y$ in a neighborhood of $0$ in $\lieg_-$, these points fill a neighborhood of $\sigma(s)$.
\end{Pf}

The following result strengthens proposition \ref{prop.curv.stab} in
the special case of holonomy paths coming from $\frak{sl}_2$--triples.
This improvement is crucial, since it provides information on the
possible values of invariant sections at the point $b_0$ rather than
at $b$.

\begin{proposition}
\label{prop.attr.curv}
In the setting of proposition \ref{prop.sl2.triple}, suppose that $X$
can be chosen so that $A=[Z,X]\in\lieg_0$. Suppose further
that $\Bbb W$ is a completely reducible representation of $P$ on which
$A$ acts diagonalizably, and that $f:B\to\Bbb W$ is the equivariant
function corresponding to a $\varphi^t_\eta$-invariant section of
$B\x_P\Bbb W$. Then $f(b_0)\in\Bbb W_{st}(A)$.
\end{proposition}
\begin{Pf}
  Set $b(t) = \exp(b_0,sX)e^{\frac{-t}{1+st} Z}$ and
  $b=\exp(b_0,sX)e^{\frac{-1}{s}Z}$, so $g(t) = e^{\log(1+st) A}$ is a
  holonomy path at $b$ with attractor $b_0$.
   Let $\mu_i$ be the eigenvalues of $A$ on $\Bbb{W}$, so
   $\lambda_i(t) = (1+st)^{\mu_i}$ are the eigenvalues of $g(t)$, for
   $i=0, \ldots, l$.
   
    Now since $\Bbb W$ is a completely reducible representation of
    $P$, the unipotent radical $P_+$ acts trivially on $\Bbb W$, and
    for all $t$,
  \begin{eqnarray*}
  f(b(t)) & = & e^{\frac{t}{1+st} Z}.f(\exp(b_0,sX)) = f(\exp(b_0,sX))  =  e^{\frac{1}{s} Z}.f(\exp(b_0,sX))\\
  & = & f(b)
  \end{eqnarray*}

  But then proposition
  \ref{prop.curv.stab} (1) simply reads as
$$
\lambda_i(t)f(b)_i\to f(b_0)_i \quad \text{ for } t\to\infty.  
$$ If $\la_i(t) \to \infty$, we know that $f(b)_i=0$ from proposition
\ref{prop.curv.stab} and thus $f(b_0)_i=0$.  We conclude $f(b_0) \in
\Bbb W_{st}(g) = \Bbb W_{st}(A)$.
\end{Pf}

\subsection*{Simplifications}  
The examples we treat below share some features that permit some simplification of these general results and lead to better descriptions
of the automorphisms.  For $Z\in\liep_+$, we
don't really use the subsets $F_\lieg(Z)$, $C_\lieg(Z)$, and
$T_\lieg(Z)$ of $\lieg$, but rather their images in $\lieg/\liep$. The
improvements are available if one can actually find subsets of
$\lieg_-$ that have the same images in $\lieg/\liep$.  Define
$F_{\lieg_-}(Z)=F_\lieg(Z)\cap\lieg_-$, and likewise for $C_{\lieg_-}$
and $T_{\lieg_-}$.

We give a proof that this simplification is always possible for $|1|$--graded
geometries and just verify the facts directly in the other examples.
For $|1|$--graded geometries, $\frak p_+=\frak
g_1$, so $Z\in\frak p_+$ is automatically homogeneous of degree one.

\begin{proposition}\label{prop.1-graded}
  Let $\lieg=\lieg_{-1}\oplus\lieg_0\oplus\lieg_1$ be a $|1|$--graded
  semsimple Lie algebra and let $Z\in\liep_+=\lieg_1$ be any element. 

  Any $X\in F_\lieg(Z)$ is congruent to an element of
  $F_{\lieg_-}(Z)$ modulo $\liep$. The analogous statements hold for
  $C_\lieg(Z)$ and $T_\lieg(Z)$. For $X\in T_{\lieg_-}(Z)$ the commutator $A=[Z,X]$ is always in $\lieg_0$.
\end{proposition}

\begin{Pf}
  Decompose $X\in\lieg$ as $X=X_{-1}+X_\liep$ according to the
  decomposition $\lieg=\lieg_{-1}\oplus\liep$. We claim that if
  $X$ lies in either $F_\lieg(Z)$, $C_\lieg(Z)$, or $T_\lieg(Z)$, then $X_{-1}$
  lies in the same subset, which clearly suffices to complete the proof.

  First, $[Z,X]=[Z,X_{-1}]+[Z,X_\liep]$ with the first summand lying
  in $\lieg_0$ and the second in $\lieg_1$. This already implies the
  claim for $C_\lieg(Z)$. Next, $[[Z,X],X]$ is congruent to
  $[[Z,X_{-1}],X_{-1}]\in\lieg_{-1}$ modulo $\frak p$. Hence if $X\in
  F_\lieg(Z)$, then we must have $[[Z,X_{-1}],X_{-1}]=0$. Since the
  condition $\ad(X_{-1})(Z)\in\frak p$ always holds, and
  $\ad(X_{-1})^k(Z)=0$ for all $k\geq 3$, we see that $X_{-1}\in
  F_\lieg(Z)$, which completes the proof for this subset.

  Since $[[Z,X],X]$ is congruent to $[[Z,X_{-1}],X_{-1}]\in\lieg_{-1}$ 
modulo $\frak p$, we also see that if $X\in T_\lieg(Z)$
  then $[[Z,X_{-1}],X_{-1}]=-2X_{-1}$. But since $\lieg_1$ is abelian,
  we also get $[[Z,X],Z]=[[Z,X_{-1}],Z]$, which completes the proof.
\end{Pf}

Next, there are obvious improvements of our basic results,
propositions \ref{prop.attractor} and \ref{prop.sl2.triple}, in the
case that one can find such nicer representatives. Suppose we have
given a Cartan geometry $(M,B,\om)$ of type $(\lieg,P)$ and
$\eta\in\alginf(M)$ with a higher order fixed point at $x_0\in M$ with
isotropy $\al\in T^*_{x_0}(M)$. Suppose that for $\xi\in
F(\al)$, there is 
$b_0\in\pi^{-1}(x_0)$ that identifies $\xi$ with some $X \in F_{\lieg_-}(Z)$, where $Z\in\frak p_+$
is the isotropy of $\eta$ with respect to $b_0$. Then the curve
$r\mapsto\exp(b_0,rX)$ in part (1) of proposition
\ref{prop.attractor} is a distinguished curve of the geometry.  If $\xi\in
C(\al)$ and we can find a representative $X\in C_{\lieg_-}(Z)$, then
one again obtains a distinguished curve in part (1) of proposition
\ref{prop.attractor}.

Finally, if $\xi\in T(\al)$ and we find a representative $X\in
T_{\lieg_-}(Z)$ then the curve $\sigma$ constructed in the proposition
\ref{prop.sl2.triple} again is a distinguished curve. 

We obtain the following formulation of these simplifications for $|1|$--graded
geometries using the normal coordinates defined in section \ref{1.1.2}. 


\begin{proposition}\label{prop.simplifications} 
  Consider a Cartan geometry $(M,B,\om)$ of type $(\lieg,P)$ corresponding
  to a $|1|$--grading
  $\lieg=\lieg_{-1}\oplus\lieg_0\oplus\lieg_1$. Let $\eta\in\alginf(M)$
  have a higher order fixed point at $x_0\in M$ with isotropy $\al\in
  T^*_{x_0}(M)$.

  Then in any normal coordinate chart centered at $x_0$, the subset $F(\al)$ consists of fixed points and
  $C(\al)\subset F(\al)$ consists of higher order fixed points of
  the same geometric type as $x_0$. If this chart has image $U \subset T_{x_0} M$, the action of the local
  flow $\phi^t_\eta$ on $U \cap \BR^* T(\al)$ is given by $\la\xi\mapsto
  \frac{\la}{1+\la t}\xi$ for $\xi\in T(\al)$ and $\lambda t > 0$.
\end{proposition}
\begin{Pf}
  Given any $b_0\in\pi^{-1}(x_0)$ we let $Z\in\lieg_1$ be the element
  corresponding to $\al$ via $b_0$. Then by proposition
  \ref{prop.1-graded} there is a representative
  $X\in\lieg_-$ for any $\xi$ in each of the subsets $C(\al),$ $F(\al)$, and $T(\al)$. Then the curves $\exp(b_0,\la X)$ used in the
  proofs of propositions \ref{prop.attractor} (3) and
  \ref{prop.sl2.triple} are the images of the line spanned by $X$
  under the inverse of the normal coordinate chart defined at
  $b_0$. Hence the claim follows from the two propositions.
\end{Pf}


\section{Applications}

\subsection{Almost-Grassmannian structures of type
  $(2,n)$}
  \label{sec.grassmann} 
 
The homogeneous model for almost-Grassmannian structures is the
Grassmannian $\mbox{Gr}(m,m+n)$, viewed as a homogeneous space
$\SL(m+n,\BR)/P$, where $P$ is the parabolic subgroup preserving an
$m$-dimensional subspace of $\BR^{m+n}$.  A normal Cartan geometry of
type $(\SL(m+n),P)$ on a manifold $M$ of dimension $mn$ is equivalent
to a first order G--structure of type $G_0 = S(\GL(m,\BR)\x GL(n,\BR)) <
GL(m + n,\BR)$. Here $G_0$ acts on the
space $\BR ^{n \times m}$ of $n\x m$--matrices by multiplications from both
sides. Explicitly, such a reduction of structure group can be
described by the following data:
 
 \begin{itemize}
 \item auxiliary vector bundles $E$ and $F$ over $M$, with fibers
   $\BR^m$ and $\BR^n$, respectively
 \item an isomorphism $TM \cong E^* \otimes F$
 \item a trivialization $\wedge^m E^* \otimes \wedge^n F \cong M
   \times \BR$
 \end{itemize}
 
More details on almost--Grassmannian structures and their description
as parabolic geometries can be found in section 4.1.3 of
\cite{cap.slovak.book.vol1}.  There are two harmonic curvature components, which will be described in some detail below.  One is a torsion, the vanishing of which is equivalent to the Cartan geometry being torsion-free.  The other harmonic component is a curvature, with values in $\lieg_0$.

The cotangent bundle $T^*M$ can be identified with $F^* \otimes E$ by
$\alpha(v) = \tr_E (\alpha\circ v)=\tr_F(v\circ\alpha)$. It is associated to the $G_0$ representation on $\BR^{m \times n}$. The geometric types of cotangent vectors are given by the ranks of the corresponding
matrices. Here we will consider almost-Grassmannian structures of type
$(2,n)$ for $n \geq 2$, so the possible ranks of non--zero cotangent
vectors are $1$ and $2$, and we will in particular prove the following result:
 
 \begin{theorem}
 \label{thm.grassmannian}
 Let $M$ be endowed with an almost-Grassmannian structure of type
 $(2,n)$, $n \geq 2$, and let $\eta \in \alginf(M)$ have a higher order fixed point at $x_0$.
 
 \begin{enumerate}
\item Higher order fixed points with isotropy of rank two are smoothly
 isolated in the strongly fixed set.  For rank one isotropy, the strongly fixed component of $x_0$ contains a submanifold of dimension $n-1$.  

\item If the
 isotropy of $\eta$ at $x_0$ has rank two, or if the geometry is torsion free, then
 there is an open set $U \subset M$ with $x_0 \in \overline{U}$ on
 which the almost-Grassmannian structure is locally flat.
 \end{enumerate}
 \end{theorem}
 
 This result will be a consequence of propositions
 \ref{prop.rank2.grassmannian} and \ref{prop.rank1.grassmannian}
 below, which give a more detailed description of the flow for each of
 the two possible geometric types.
 
 The Lie algebra $\mathfrak{sl}(n+2,\BR)$ has a $|1|$--grading coming
 from the block decomposition
$$
\begin{pmatrix}
  \lieg_0 & \lieg_1\\ \lieg_{-1} & \lieg_0
\end{pmatrix}
$$
with block sizes $2$ and $n$. Here $\lieg_0=\mathfrak{s}(\alggl(2,\BR)
\times \alggl(n,\BR))$ and, as $\lieg_0$--modules, we have
$\lieg_{-1}=L(\BR^2,\BR^n)$ and $\lieg_1=L(\BR^n,\BR^2)$. The dual
pairing between $\lieg_{-1}$ and $\lieg_1$ is given by $(X,Z)\mapsto
\tr(Z X)=\tr(X Z)$, while the bracket of these two is $[Z,X]=(Z X,-X
Z) \in \lieg_0$.  Further,
$$ [[Z,X],X] = -2XZX \qquad \mbox{and} \qquad [[Z,X],Z] = 2 ZXZ$$

In terms of the standard
representations $\BR^2$ and $\BR^n$ of the two factors of
$\lieg_0$, we have
$\lieg_{-1}\cong\BR^{2*}\otimes \BR^n$ and
$\lieg_1\cong\BR^2\otimes\BR^{n*}$; consequently,
$$
\Lambda^2\lieg_1=(\Lambda^2\BR^2\otimes S^2\BR^{n*})\oplus (S^2 \BR^2
\otimes\Lambda^2\BR^{n*}).
$$
We will apply the general results of section \ref{sec.general} to the
harmonic curvatures,
which, for $n \geq 3$, are the highest weight components
$$
\Bbb V \subset (S^2\BR^2\otimes\Lambda^2\BR^{n*})\otimes(\BR^{2*}\otimes\BR^n)\subset
\Lambda^2\lieg_1\otimes\lieg_{-1},
$$
and 
$$
\Bbb U \subset (\Lambda^2\BR^2\otimes S^2\BR^{n*})\otimes\algsl(n,\BR)\subset
\Lambda^2\lieg_1\otimes\lieg_0. 
$$
The first part will be called the ``harmonic torsion,'' denoted $\tau$
below, and the second the ``harmonic curvature,'' denoted $\rho$.  By
the general theory, vanishing of the harmonic torsion is equivalent to
torsion freeness, while vanishing of both the harmonic curvature and
the harmonic torsion on an open subset is equivalent to local
flatness.  We will write $\Bbb V^1 = S^2 \BR^2 \otimes \BR^{2*} $ and $\Bbb V^2 = \Lambda^2 \BR^{n*} \otimes \BR^n$ below, so $\Bbb V \subset \Bbb V^1 \otimes \Bbb V^2$.

Almost-Grassmannian structures of type $(2,2)$ are equivalent to
4-dimensional conformal spin structures in split signature $(2,2)$ via
the isomorphisms $G = \SL(4, \BR) \cong \mbox{Spin}(3,3)$ and $G_0 =
S(\GL(2,\BR) \times \GL(2,\BR)) \cong \mbox{CSpin}(2,2)$. The tangent bundle
naturally can be written as the tensor product of the two basic real
spinor bundles on 4-dimensional conformal manifolds of split
signature.  Almost Grassmannian
structures of type $(2,n)$ provide a natural higher dimensional analog
of this so--called spinor formalism in 4-dimensional conformal geometry. For type $(2,2)$, the harmonic curvature still
consists of two components, but there are two curvatures instead of
one torsion and one curvature, the self--dual and anti--self--dual
parts of the Weyl curvature. Both curvatures have values in bundles
associated to a highest weight subspace in $(\Lambda^2\BR^2\otimes
S^2\BR^{2*})\otimes\algsl(2,\BR)$, of which there are two.  We will assume $n\geq 3$
below, but our arguments also apply to the case $n=2$,
and thus provide another proof for split signature conformal
structures in dimension $4$.


\subsubsection{The rank two case} We start by collecting the algebraic
results.
\begin{lemma}\label{lem:Grass-two-alg}
Suppose that $Z\in L(\BR^n,\BR^2) \cong \lieg_1$ has rank two. Then 
\begin{enumerate}
\item  The subspaces associated to $Z$ are
\begin{gather*}
  \{0\}=C_{\lieg_-}(Z)\subset F_{\lieg_-}(Z)=\{X\in
  L(\BR^2,\BR^n):XZX=0\}\\
  T_{\lieg_-}(Z)=\{X\in L(\BR^2,\BR^n) \ : \ ZX=\mathrm{Id}_{\BR^2}\} \\
  \cong \{ W \subset \BR^n \ : \ \dim W = 2, \ W \cap \ker(Z) = \{ 0 \} \}
\end{gather*}
where the last isomorphism is induced by $X \mapsto \im(X)$.

\item For $X\in T_{\lieg_-}(Z)$, let $A=[Z,X]$
and $W=\im(X)$. Then the eigenvalues of $A$ on
$\lieg_{-1}$ are all negative. For the representation
$\Bbb V$
corresponding to the harmonic torsion,  $\Bbb V_{ss}(A)=0$, and
$$
\Bbb V_{st}(A)\subset
(S^2\BR^2\otimes\Lambda^2\BR^{n*})\otimes(\BR^{2*}\otimes W).
$$
For the representation $\Bbb U$
corresponding to the harmonic curvature,  $\Bbb U_{st}(A)=\{0\}$.
\end{enumerate}
\end{lemma}

\begin{Pf}
  By assumption, $Z$ is onto and $\ker(Z)\subset\BR^n$ is a subspace
  of dimension $n-2$. From the brackets computed above, the
  descriptions of $C_{\lieg_-}(Z)$ and $F_{\lieg_-}(Z)$ follow
  immediately. 

For $X \in L(\BR^2,\BR^n)$, the condition $[[Z,X],Z]=2Z$ implies $ZXZ=Z$, so any $X\in
  T_{\lieg_-}(Z)$ has rank two and
  $\im(X)\cap\ker(Z)=\{0\}$.  For any two--dimensional subspace
  $W\subset\BR^n$ complementary to $\ker(Z)$, the map $Z$
  restricts to a linear isomorphism $W\to\BR^2$; if $\im(X)=W$,
  then $ZXZ=Z$ if and only if $X$ is the
  inverse of this isomorphism. In this case, $ZX=\id_{\BR^2}$ and $[[Z,X],X]=-2X$ follows automatically. Hence
  $$T_{\lieg_-}(Z) = \{ W \subset \BR^n \ : \ \dim W = 2, \ W \cap \ker Z = \{ 0 \} \}$$
  via $X=(Z|_W)^{-1}$. We also see
  immediately that $A=[Z,X]$ acts by $1$ on $\BR^2$, while on $\BR^n$, the eigenspaces are $W$ and $\ker(Z)$ with eigenvalues
  $-1$ and $0$, respectively. Thus all eigenvalues of $A$ on $\lieg_- \cong L(\BR^2,\BR^n)$ are negative. 
  
  To analyze the representation
$$
(S^2\BR^2\otimes\Lambda^2\BR^{n*})\otimes(\BR^{2*}\otimes\BR^n),
$$ 
observe that $A$ acts as the identity on
$S^2\BR^2\otimes\BR^{2*}$. The eigenspaces in $\BR^{n*}$ are $W^o$ and $\ker(Z)^o$, the respective annihilators of $W$ and $\ker(Z)$, with eigenvalues $0$ and $1$,
respectively. Hence the eigenspace structure on $\Lambda^2\BR^{n*}\otimes\BR^n$ is:
\begin{center}
\begin{tabular}{| c | c |}
\hline
{\bf subspace} & {\bf eigenvalue} \\
\hline
$ \Lambda^2 \ker(Z)^o \otimes \ker(Z)$ & $2$ \\
$ \Lambda^2  \ker(Z)^o \otimes W \oplus (\ker(Z)^o\wedge W^o) \otimes \ker(Z)$ & $1$ \\
$ (\ker(Z)^o \wedge W^o) \otimes W \oplus \Lambda^2 W^o\otimes \ker(Z)$ & $0$ \\
$ \Lambda^2 W^o \otimes W$ & $-1$ \\
\hline
\end{tabular}
\end{center}
The only nonnegative eigenvalue on $\Bbb V$ is 0, and the claim $\Bbb V_{st} = S^2 \BR^2 \otimes \Lambda^2 W^o \otimes \BR^{2*} \otimes W$ follows. 

Next consider
$$
(\Lambda^2\BR^2\otimes S^2\BR^{n*})\otimes\algsl(n,\BR).
$$
Here $A$ acts by multiplication by $2$ on $\La^2\BR^2$, while on
$S^2\BR^{n*}$ the possible eigenvalues are $0$, $1$, and $2$. Finally,
on $\algsl(n,\BR)$ the possible eigenvalues evidently are $-1$, $0$,
and $1$.  Thus all eigenvalues on $\Bbb U$ are positive.
\end{Pf}

These results can be immediately translated to geometry:

\begin{proposition}\label{prop.rank2.grassmannian}
Let $M$ be endowed with an almost-Grassmannian structure of type
$(2,n)$, $n \geq 2$, and let $\eta \in \alginf(M)$ vanish to higher order at $x_0 \in M$.
Assume that the isotropy of $\eta$ is $\alpha \in T_{x_0}^*M \cong
L(F_{x_0}, E_{x_0})$ of rank two. 
\begin{enumerate}
\item The subsets $C(\alpha) \subset F(\alpha)$ of $T_{x_0} M \cong L(F_{x_0}, E_{x_0})$ are 
$$ \{ 0 \} = C(\alpha) \subset F(\alpha ) = \{ \xi : \ \xi\o\al\o\xi=0 \}$$
and 
$$ T(\alpha) = \{ \xi \ : \ \alpha \o \xi = \mathrm{Id}_{E_{x_0}} \}$$

\item In any normal coordinate chart
centered at $x_0$, there is an open neighborhood $U$ of $0$ such that
\begin{enumerate}
\item The higher order fixed point $x_0$ is smoothly isolated in the strongly fixed set of the flow, and elements of $F(\alpha) \cap U$ are fixed points.

\item For any $\xi\in U$ in the cone 
$$ S = \{ \xi \ : \ \al\o\xi = \lambda \mathrm{Id}_{F_{x_0}} \ \mbox{for some} \ \lambda \neq 0 \}$$
the action of the flow in normal coordinates is
$$\phi^t_{\eta}(\xi)=\frac{1}{2+t \cdot {\mathrm{tr}}(\al\o\xi)} \cdot \xi \qquad \mbox{for} \ t \cdot {\mathrm{tr}}(\alpha \circ \xi) > 0$$
\end{enumerate}

\item There is an open neighborhood of $S\setminus\{0\}$ on which the
geometry is locally flat.
\end{enumerate}
\end{proposition}

\begin{Pf}
  From lemma \ref{lem:Grass-two-alg}, it is clear that $\xi\in
  T_{x_0}M$ lies in $F(\al)$ if and only if
  $\xi\o\al\o\xi=0$. On the other hand, $\xi$ is a multiple of an
  element $\xi_0\in T(\al)$ if and only if $\al\o\xi$ is a multiple of
  the identity, and then $\xi=\tfrac{\tr(\al\o\xi)}{2} \cdot \xi_0$.  Using
  this, (1) and (2) follow directly from propositions
  \ref{prop.attractor} and \ref{prop.simplifications}.

  To prove part (3), let $\xi\in T(\al)$, and apply proposition
  \ref{prop.sl2.triple} to the line $l$
  spanned by $\xi$. Now corollary \ref{cor:sl2-path} together
  with lemma
  \ref{lem:Grass-two-alg} shows that the harmonic curvature $\rho$
  vanishes on a neighborhood of $l\setminus\{0\}$.  Thus this curvature vanishes on a neighborhood of $S \backslash \{0\}$.

  For the harmonic torsion, apply proposition
  \ref{prop.attr.curv} to the description of $\Bbb
  V_{ss}(A)$ in lemma \ref{lem:Grass-two-alg} to see that
  for any $u,v \in T_{x_0} M$, the harmonic torsion $\tau_{x_0}(u,v) \in T_{x_0}M\cong L(E_{x_0},F_{x_0})$ has values in
  $\im(\xi)\subset F_{x_0}$. The same holds for any $\xi$; from
  lemma \ref{lem:Grass-two-alg}, $\im(\xi)$ can be any two--dimensional
  subspace complementary to $\ker(\al)\subset F_{x_0}$.  The intersection of all such subspaces is 0, so $\tau(x_0)=0$.  Now corollary
  \ref{cor:sl2-path} (3) implies that $\tau$ vanishes on an open
  neighborhood of $S\setminus\{0\}$. Thus all harmonic curvature
  components vanish locally around $S\setminus\{0\}$, so the geometry
  is flat on an open set as claimed.
\end{Pf}

\subsubsection{The rank one case}
\begin{lemma}\label{lem:Grass-one-alg}
Let $Z\in L(\BR^n,\BR^2)=\lieg_1$ be of rank one.

\begin{enumerate} 
\item The sets associated to $Z$ are
  \begin{gather*}
    C_{\lieg_-}(Z)=\{X\in
    L(\BR^2,\BR^n): {\mathrm{im}}(Z) \subset \ker(X) \ \text{and} \ \mathrm{im}(X)\subset\ker(Z)\}\\ 
    F_{\lieg_-}(Z)=\{X\in
    L(\BR^2,\BR^n):XZX=0\}\\ T_{\lieg_-}(Z)=\{X\in L(\BR^2,\BR^n):
    \mathrm{rk}(X)=1, \mathrm{tr}(ZX) = \mathrm{tr}(XZ) =1\}.
  \end{gather*}

\item For any choice of lines $V\subset\BR^2$ transversal to
  $\im(Z)$ and $W\subset\BR^n$ transversal to $\ker(Z)$, there is a
  unique element $X\in T_{\lieg_-}(Z)$ with $\ker(X)=V$ and
  $\im(X)=W$.

\item Let $X$, $V$, and $W$ be as in (2).  Let $V^\o\subset\BR^{2*}$ and
  $W^\o\subset\BR^{n*}$ be the annihlators, and let $A=[Z,X]$. Then
  all eigenvalues of $A$ on $\lieg_{-1}$ are non--positive, and the
  $0$--eigenspace coincides with $C_{\lieg_-}(Z)$. Moreover, for the
  representations $\Bbb V \subset \Bbb V^1 \otimes \Bbb V^2$ corresponding to the harmonic torsion and $\Bbb U$
  corresponding to the harmonic curvature, 
  
  \begin{enumerate}
\item $\Bbb V^1_{ss}(A)=S^2 V\otimes V^\o$ 

\item $\Bbb V^1_{st}(A)\subset S^2V\otimes\BR^{2*} + (V \odot \BR^2)\otimes V^\o$

\item $\Bbb V^2_{ss}(A)=\La^2W^\o\otimes W$ 

\item $\Bbb V^2_{st}(A)\subset  \La^2W^\o\otimes\BR^n+(W^\o\wedge\BR^{n*})\otimes W$

\item  $\Bbb V_{ss}(A)\subset \Bbb V^1_{ss}(A)\otimes\Bbb V^2_{st}(A)+\Bbb
  V^1_{st}(A)\otimes\Bbb V^2_{ss}(A)$
  
\item $\Bbb V_{st}(A)\cap  ((S^2\BR^2\otimes\Lambda^2\BR^{n*})\otimes C_{\lieg_-}(Z)) \\  \subset(S^2V\otimes \La^2W^\o)\otimes C_{\lieg_-}(Z)$

\item $\Bbb U_{ss}(A)=\{0\}$

\item $\Bbb U_{st}(A)\subset (\Lambda^2\BR^2\otimes S^2\BR^{n*})\otimes\BR^{n*}\otimes W$
\end{enumerate}
\end{enumerate}
\end{lemma}

\begin{Pf}
(1) Since $Z$ has rank one, $\ker(Z)\subset\BR^n$ has dimension $n-1$ and
$\im(Z)\subset\BR^2$ is one--dimensional. Now $X\in \lieg_{-1}
\cong L(\BR^2,\BR^n)$ lies in $C_{\lieg_-}(Z)$ if and only if
$ZX=0$ and $XZ=0$, that is, if and only if $\im(Z)\subset\ker(X)$ and
$\im(X)\subset\ker(Z)$. For $X \neq 0$, this means
$\ker(X)=\im(Z)$, so $X$ has rank one, and we obtain the description of
$C_{\lieg_-}(Z)$. The description of $F_{\lieg_-}(Z)=\{X:XZX=0\}$
follows exactly as in the proof of lemma \ref{lem:Grass-two-alg}.

Next, $X\in T_{\lieg_-}(Z)$ is evidently equivalent to $XZX=X$ and
$ZXZ=Z$. Since $\rk(XZX) \leq 1$, the first equality implies $\rk(X)= 1$.  Also, $(ZX)^2=ZX$, so $ZX$ is a rank one projection with
$\tr(ZX)=1$. Conversely, suppose $\rk(X) = 1$ and $\tr(ZX)=1$.  Then
$\im(X)$ is transversal to $\ker(Z)$ and $\im(Z)$ is transversal to
$\ker(X)$. Taking a basis adapted to the splitting
$\BR^2=\ker(X)\oplus\im(Z)$, we conclude that $ZX$ is the projection
onto the second factor, so $ZXZ=Z$ and $XZX=X$.  Now (1) is proved.

(2) Given lines $W$ and $V$ as in the statement, there is a unique linear map
$X:\BR^2\to\BR^n$ with $\ker(X)=V$ and $\im(X)=W$ up to scale. Now $Z$
induces a linear isomorphism $W\to \BR^2/V$, and the remaining condition $\tr(ZX) = 1$ for
$X\in T_{\lieg_-}(Z)$ is equivalent to $X$ inducing the
inverse of this isomorphism. 

(3) Let $X$ be as in (2), so $ZX$ is the projection onto the first
factor of $\BR^2=\im(Z)\oplus V$, while $XZ$ is the second projection
in $\BR^n=\ker(Z)\oplus W$. Hence $A=[Z,X]$ has eigenvalues $1$ on
$\im(Z)$ and $0$ on $V = \ker(X)$, and eigenvalues $0$ on $\ker(Z)$
and $-1$ on $W$. The eigenspace decompositions $\BR^{n*}=W^\o\oplus
\ker(Z)^\o$, with respective eigenvalues $0$ and $1$, and
$\BR^{2*}=V^\o\oplus \im(Z)^\o$, with respective eigenvalues $-1$ and
$0$, follow, which implies all our claims on eigenvalues and
eigenspaces for $\lieg_{-1} \cong L(\BR^2,\BR^n)$.

 Next, it is easy to compute the eigenspace decomposition of $\Bbb V^1 \cong S^2 \BR^2 \otimes \BR^{2*}$.
 \begin{center}
 \begin{tabular}{| c | c |}
 \hline
 {\bf subspace} & {\bf eigenvalue} \\
 \hline
 $S^2(\im(Z)) \otimes \im(Z)^o$ & $2$ \\
 $S^2(\im(Z)) \otimes V^o \oplus (\im(Z) \odot V) \otimes \im(Z)^o$ & $1$ \\
 $(\im(Z) \odot V) \otimes V^o \oplus S^2(V) \otimes \im(Z)^o$ & $0$ \\
  $S^2(V) \otimes V^o$ & $-1$ \\
  \hline
  \end{tabular}
  \end{center}
  
 and of $\Bbb V^2 \cong \Lambda^2\BR^{n*}\otimes\BR^n$:
\begin{center}
\begin{tabular}{| c | c |}
\hline
{\bf subspace} & {\bf eigenvalue} \\
\hline
$ \Lambda^2 \ker(Z)^o \otimes \ker(Z)$ & $2$ \\
$ \Lambda^2  \ker(Z)^o \otimes W \oplus (\ker(Z)^o \wedge W^o) \otimes \ker(Z)$ & $1$ \\
$ (\ker(Z)^o \wedge W^o) \otimes W \oplus \Lambda^2 W^o\otimes \ker(Z)$ & $0$ \\
$ \Lambda^2 W^o \otimes W$ & $-1$ \\
\hline
\end{tabular}
\end{center}

Now the claims (a)--(d) on $\Bbb V^i_{ss}(A)$ and $\Bbb
V^i_{st}(A)$, for $i=1,2$, can be read from the tables.

To get a negative eigenvalue on $\Bbb V^1\otimes\Bbb V^2$ one has to
have eigenvalue $-1$ on one factor and a non--positive eigenvalue on
the other, which proves the claim (e) on $\Bbb V_{ss}(A)$. 
Since $C_{\lieg_-}(Z)$ is the zero eigenspace for $A$ in $\lieg_{-1}$, and all eigenvalues on $S^2 \BR^2$ and $\Lambda^2 \BR^{n*}$ are non--negative, a
non--positive eigenvalue on $(S^2\BR^2\otimes\La^2\BR^{n*})\otimes
C_{\lieg_-}(Z)$ is only possible if the eigenvalue on the first factor
is zero. This implies the claim (f) on $\Bbb V_{st}(A)$.   

To deal with $\Bbb U$, observe that on $\Lambda^2 \BR^2$, $A$ acts as
the identity, while the possible eigenvalues on $S^2 \BR^{n*}$ range
between $0$ and $2$. On $\algsl(n,\BR)$ the eigenspace decomposition is

\begin{center}
\begin{tabular}{| c | c |}
\hline
{\bf subspace} & {\bf eigenvalue} \\
\hline
$W^o \otimes W$ & $-1$ \\
$\ker(Z)^o \otimes W \oplus W^o \otimes \ker(Z)$ & $0$ \\
$\ker(Z)^o \otimes \ker(Z)$ & $1$ \\
\hline
\end{tabular}
\end{center}

The $(-1)$--eigenspace consists of
maps having values in $W$, so the claims (g) and (h)
 that $\Bbb U_{ss}(A) = 0$ and $\Bbb U_{st}(A)\subset  (\Lambda^2\BR^2\otimes S^2\BR^{n*})\otimes\BR^{n*}\otimes W$ are proved. 
%
\end{Pf}

 A linear map $X:\BR^2\to\BR^n$ of rank one with
$\ker(X)=V$ and $\im(X)=W$ naturally determines an
$(n+1)$--dimensional subspace $\frak a(X)\subset L(\BR^2,\BR^n)$, where
$$ X \in \mathfrak{a}(X) = \{ Y \ : \ \im(Y) \subseteq W \} \ + \ \{ Y  \ : \ V \subseteq \ker(Y) \}$$\

 The intersection of the two summands is the line spanned by
$X$. For $\xi \in T_xM$ of rank one, denote by $\mathfrak{a}(\xi) \subseteq T_xM$ the corresponding subspace.


\begin{proposition}\label{prop.rank1.grassmannian} 
  Let $M$ be endowed with an almost-Grassmannian structure of type
  $(2,n)$, $n \geq 2$, and let $\eta \in \alginf(M)$ vanish to higher order at $x_0 \in M$.
  Assume that the isotropy $\alpha\in T_{x_0}^*M \cong
  L(F_{x_0},E_{x_0})$ of $\eta$ has rank one. 

\begin{enumerate}
\item The subsets $C(\al)\subset F(\al)$ of $T_{x_0}M\cong
  L(E_{x_0},F_{x_0})$ are given by
$$
    \{\xi:\al\o\xi=\xi\o\al=0\}=C(\al)\subset
    F(\al)=\{\xi:\xi\o\al\o\xi=0\}.
$$
Moreover, for each choice of lines $V\subset E_{x_0}$ transversal to $\im(\al)$
and $W\subset F_{x_0}$ transversal to $\ker(\al)$, there is a unique
element $\xi\in T(\al)$ with $\ker(\xi)=V$ and $\im(\xi)=W$. 

\item For any normal coordinate chart centered at $x_0$, there is an open neighborhood $U$ of $0$ such that

\begin{enumerate}
\item Elements of $F(\al)\cap U$ are fixed points and $C(\al)\cap
  U$ lies in the strongly fixed component of $x_0$.
  
\item For any rank one element $\xi\in U$ such that
  $\al(\xi) \neq 0$, the flow acts in normal coordinates by
  $\ph^t_{\eta}(\xi)=\frac{1}{1+t\al(\xi)} \cdot \xi$, whenever $t \alpha(\xi) > 0$.
\end{enumerate}

\item Both the harmonic torsion $\tau$ and the harmonic curvature $\rho$ vanish on $C(\al)\cap U$.

\item Let $\xi$ be as in (2)(b) above, and let $c$ be the distinguished curve obtained from the line spanned by $\xi$.  Then
\begin{enumerate}
\item There is a neighborhood $U_0$ of $c\setminus\{ x_0 \}$ in $M$ on which the harmonic
curvature $\rho$ vanishes identically. In particular, if the geometry
is torsion free, then it is flat on $U_0$. 

\item The harmonic torsion $\tau\in\Om^2(M,TM)$ has the following properties along $c \setminus \{ x_0 \}$:
\begin{enumerate}
\item $i_\eta(\tau) = 0$, and $\tau_x(u,v) \in \frak{a}(\eta(x))$ for any $u,v \in T_xM$.
\item If $\xi\in\frak a(\eta(x))$, then $i_\xi\tau_x$ vanishes on $\frak a(\eta(x))$ and has values in $\BR \eta(x)$.
\end{enumerate}
Moreover, $\tau$ has the same algebraic type on a neighborhood of $c \setminus \{ x_0 \}$.  
\end{enumerate}
\end{enumerate}
\end{proposition}

\begin{Pf}
  Point (1) immediately follows from lemma \ref{lem:Grass-one-alg}.  Then part (a) of (2) follows from
  proposition \ref{prop.simplifications}. The elements $\xi$ in part
  (b) of (2) are exactly the positive multiples of elements
  $\xi_0\in T(\al)$, and $\xi=\al(\xi)\xi_0$. Thus (2)(b)
  also follows from proposition \ref{prop.simplifications}.

 For part (3), first consider the harmonic torsion and curvature at the
 point $x_0$.  Let $\xi \in T(\alpha)$.  By proposition \ref{prop.attr.curv} 
and lemma \ref{lem:Grass-one-alg}, $\rho_{x_0}(u,v)$, for any $u,v \in T_{x_0} M$, is an endomorphism of $F_{x_0}$
with values in $\im(\xi)$.  Varying $\xi \in T(\alpha)$ gives all lines transversal to $\ker(\al)$, so $\rho_{x_0}=0$. 

 To show that $\tau(x_0)=0$, use the observation in section 2.3 of
  \cite{cap.srni04} that for the value
  $\ka(b_0)\in\La^2\lieg_1\otimes\lieg$ of the curvature function and
  for $Z\in\lieg_1$ corresponding to $\eta(x_0)$ via $b_0$ as above, $Z\cdot\ka(b_0)=0$, where the action is induced by the adjoint
  action. (This fact is a simple consequence of vanishing of the Lie derivative $L_{\tilde \eta} \ka$ in the Cartan bundle $B$.)  
Now decompose $\ka(b_0)$ according to
 $\lieg=\lieg_{-1}\oplus\lieg_0\oplus\lieg_1$. The component
 $\ka_{-1}(b_0)$ is well known to be harmonic, because it is the
 lowest homogeneous component, so it represents the harmonic torsion
 $\tau$. Since the adjoint action is compatible with the gradings, we have $Z\cdot\ka_{-1}(b_0)=0$. Now $Z$ commutes with
 $\lieg_1$, so this condition just means that all values $\tau_{x_0}(u,v)$ are in $C(\al)\subset T_{x_0}M$.

 On the other hand, proposition \ref{prop.attr.curv} gives that
 for any $\xi\in T(\al)$, the function corresponding to $\tau$ must
 have values in $\Bbb{V}_{st}(A)$. In particular, by part (3)(f) of lemma \ref{lem:Grass-one-alg}, 
 $\tau(x_0)$, viewed as a skew symmetric bilinear map $\BR^n\x\BR^n\to
 S^2\BR^2\otimes C(\al)$, vanishes whenever one of the arguments
 is in $W$. Now varying $\xi$ gives different $W$ spanning all of $\BR^n$, so we conclude
 that $\tau(x_0)=0$. By part (a) of (2), $C(\al)$ in normal
 coordinates corresponds to higher order fixed points with rank one
 isotropy, so (3) follows.


  (4) By lemma \ref{lem:Grass-one-alg}, any $\xi \notin \ker(\alpha)$ gives rise to $A$ for which $\Bbb U_{ss}(A)=\{0\}$. Thus the statement on $\rho$ follows from (3) above and part
  (3) of corollary \ref{cor:sl2-path}.  If the geometry is torsion
  free, then all harmonic curvatures vanish on an open neighborhood of
  $c \setminus \{ x_0 \}$, so the geometry is flat there.

To prove (4)(b), let $\tilde{c}(s) = \exp(b_0,sX)$ for $X \in T_{\lieg_-}(Z)$, so $c = \pi \circ \tilde{c}$.  Because $\eta$ preserves $c$, it is tangent along $c$.
Recall the notation of lemma \ref{lem:Grass-one-alg}:  $X \in V^\o \otimes W$ and $\frak{a}(X) = V^\o\otimes\BR^n+\BR^2\otimes
W$.  Let $\tilde{\tau}$ be the $P$-equivariant function on $B$ corresponding to the harmonic torsion.  By part (3) immediately above, $\tilde{\tau}(b_0) = 0$, so by proposition \ref{prop.curv.stab}, $\tilde{\tau} \in \Bbb{V}_{ss}$ along $\tilde{c}$.
This subspace is identified in part (3)(e) of lemma \ref{lem:Grass-one-alg}.  The first term, $\Bbb{V}^1_{ss} \otimes \Bbb{V}^2_{st}$, corresponds to two-forms, the values of which are endomorphisms that vanish on $V$.  The second term, $\Bbb{V}^1_{st} \otimes \Bbb{V}^2_{ss}$, corresponds to two-forms, the values of which are endomorphisms with image in $W$.  Thus $\tau_x(u,v) \in \frak{a}(\eta(x))$ for $x \in c \setminus \{x_0 \}$, as claimed.


From the description of $\Bbb{V}_{ss}$, one can also see that $i_{\eta}\tau=0$ and that $\tau_x(u,v)$ vanishes when
$u,v \in \frak a(\eta(x))$. Finally, inserting an element of $V^\o$ into $\Bbb V^1_{st}(A)$ gives
an element of $V\otimes V^\o$, while inserting an element of $W$ into $\Bbb
V^2_{st}(A)$ gives an element of $W^\o\otimes W$, and the
last remaining claim follows. 
\end{Pf}
%
%

\subsection{Almost Quaternionic structures}
A quaternion--K\"ahler metric on a smooth
manifold has an underlying almost quaternionic structure, which is
automatically integrable. These integrable structures are often
referred to as \textit{quaternionic structures}, and the integrability
is equivalent to torsion freeness of the associated Cartan geometry.
Almost quaternionic structures are very similar to almost Grassmannian structures of
type $(2,2n)$, since the Lie algebras governing the two geometries
have the same complexification. As in the Grassmanian case, there is a
relation to conformal geometry in the lowest dimensional case
$n=1$. This time, however, the corresponding conformal structures are
not of split signature but of definite signature. In view of this
close analogy, we can carry over most of the results from the
Grassmannian case rather easily and treat this geometry quite briefly.

The homogeneous model is the quaternionic projective space
${\bf HP}^n$, viewed as a homogeneous space of $PGL(n+1,\BH)$. The
corresponding geometries can be described as first order structures
with structure group $G_0 =S(\BH^* \x GL(n,\BH)) <
GL(4n,\BR)$, with factors corresponding to scalar
multiplications by non--zero quaternions (which are not
quaternionically linear maps since $\BH$ is
non--commutative) and quaternionically linear automorphisms.  Such a geometry on a smooth manifold
$M^{4n}$ is given by a rank three subbundle $\Cal
Q\subset \mbox{End}(TM)$, locally spanned by $I$, $J,$ and $IJ$
for two anti--commuting almost complex structures $I$ and $J$. See
section 4.1.8 of \cite{cap.slovak.book.vol1} for more details on these
geometries.

The Lie algebra governing the geometry is $\lieg=\frak{sl}(n+1,\BH)$
endowed with a $|1|$--grading
$\lieg=\lieg_{-1}\oplus\lieg_0\oplus\lieg_1$ given by blocks of size
$1$ and $n$ in quaternionic matrices. Thus $\lieg_{-1}\cong
L_{\BH}(\BH,\BH^n)$, $\lieg_0\cong\frak
s(\frak{gl}(1,\BH)\oplus\frak{gl}(n,\BH))$ and $\lieg_1\cong
L_{\BH}(\BH^n,\BH)$. In particular, there is just one non--zero orbit
of $G_0$ in $\lieg_1$, which behaves as in the rank two case for
Grassmannian structures. The spaces $\La^2\lieg_{\pm 1}$ and
$S^2\lieg_{\pm 1}$ can be decomposed according to their quaternionic
linearity properties. The relevant torsion and curvature components
are specified in terms of these subspaces, see section 4.1.8 of
\cite{cap.slovak.book.vol1}.  Now the following results can be proved
analogously to lemma \ref{lem:Grass-two-alg}. There is a small
simplification in the description of $F_{\lieg_-}(Z)$, however, coming
from the fact that any nonzero $X\in L_\BH(\BH,\BH^n)$ is injective.

\begin{lemma}\label{lem:Quat}
Let $Z\in\lieg_1\cong L_{\BH}(\BH^n,\BH)$ be any nonzero element. Then
\begin{enumerate}
\item The subspaces determined by $Z$ are
\begin{gather*}
  \{0\}=C_{\lieg_-}(Z)\subset F_{\lieg_-}(Z)=\{X\in L_{\BH}(\BH,\BH^n):ZX=0\}\\ 
  T_{\lieg_-}(Z)=\{X\in  L_{\BH}(\BH,\BH^n): ZX=\mathrm{Id}_{\BH}\} 
\end{gather*}
and mapping $X$ to $im(X)$ identifies $T_{\lieg_-}(Z)$ with the space of
all quaternionic lines in $\mathbf H^n$ having zero intersection
with $ker(Z)$. 

\item  For $X\in T_{\lieg_-}(Z)$, let $A=[Z,X]$
and $W=\im(X)\subset {\bf H}^n$. Then the eigenvalues of $A$ on
$\lieg_{-1}$ are all negative. Denoting by $\Bbb V$ and $\Bbb U$ the
representations of $G_0$ corresponding to the harmonic torsion and the
harmonic curvature, we have $\Bbb V_{ss}(A)=\{0\}$ and $\Bbb
U_{st}(A)=\{0\}$. Finally,
$$
\Bbb V_{st}(A)\subset \La^2\lieg_1\otimes L_{\BH}(\BH,W)\subset
\La^2\lieg_1\otimes\lieg_{-1}. 
$$
\end{enumerate}
\end{lemma}

The translation to geometry is also closely similar to the case of
Grassmannian structures.  The natural structure on $T_{x_0}M$ is the three dimensional
subspace $\Cal Q_{x_0}\subset L(T_{x_0}M,T_{x_0}M)$ spanned by two anti--commuting almost complex structures and their
product; a tangent vector $\xi\in T_{x_0}M$ determines a quaternionic subspace of $T_{x_0}M$ generated by $\xi$, which
by definition is spanned by $\xi$ and the elements $q(\xi)$ for
$q\in\Cal Q_{x_0}$. Using this structure, the following result is
proved precisely as proposition \ref{prop.rank2.grassmannian}. 

\begin{theorem}
\label{thm.quaternion}
Let $M$ be a smooth manifold of dimension $4n$ endowed with an almost
quaternionic structure $\Cal Q\subset L(TM,TM)$. Let $\eta \in
\alginf(M)$, vanishing to higher order at $x_0 \in M$ with isotropy $\alpha \in
T_{x_0}^*M$. 
\begin{enumerate}
\item The sets $C(\alpha) \subset F(\alpha) \subset T_{x_0} M$ are 
$$ \{ 0 \} = C(\alpha) \subset F(\alpha) = \{ \xi \ : \ \alpha(\xi) = 0 \ \mbox{and} \ \alpha(q(\xi)) = 0 \ \forall q \in \Cal Q_{x_0} \}$$
and 
$$ T(\alpha) = \{ \xi \ : \ \alpha(\xi) = 1 \ \mbox{and}\  \alpha(q(\xi)) = 0 \ \ \  \forall q \in \Cal Q_{x_0}\}$$

\item In any normal coordinate chart centered at $x_0$, there is an open neighborhood $U$ of $0$ such that
\begin{enumerate}
\item The higher order fixed point $x_0$ is smoothly isolated in the strongly fixed set of the flow, and points of $F(\alpha) \cap U$ are fixed.

\item On the cone 
$$ S = \{ \xi\in U \ : \ \al(\xi) \neq 0, \ \al(q(\xi))= 0 \ \forall \ q\in\Cal Q_{x_0} \},$$
 the flow in normal coordinates acts by
$$ \phi^t_\eta(\xi)=\frac{1}{1+t\al(\xi)} \cdot \xi \qquad \mbox{for}\  \ t \alpha(\xi) > 0$$
\end{enumerate}

\item There is an open neighborhood of $S\setminus\{0\}$ on which the
geometry is locally flat.
\end{enumerate}
\end{theorem}

\subsection{A general result for parabolic contact structures}
\label{sect:contact} 

These geometries are associated to \textit{contact
  gradings}, $|2|$--gradings such that $\lieg_{-2}$ has dimension
one and the Lie bracket $[\ ,\ ]:\lieg_{-1}\x\lieg_{-1}\to\lieg_{-2}$
is a non--degenerate bilinear form.  Such gradings
can only exist on simple, not on semisimple, Lie algebras; moreover, on
each complex simple Lie algebra and most non--compact real forms,
there is a unique grading of this type. On the manifold $M$, the subspace $\lieg_{-1}$
corresponds to a distribution $T^{-1} M$ of corank one, and the condition on the
bracket exactly means that, for a regular parabolic geometry, the distribution is contact. The subalgebra $\lieg_0\subset
L(\lieg_{-1},\lieg_{-1})$ gives an additional structure on the
contact distribution. The best known example of these geometries
is partially integrable almost CR structures, for which this additional structure is an almost complex structure on the contact
distribution. CR structures will be discussed in more detail in subsection
\ref{sect:CR} below.

For parabolic contact structures, having a higher order fixed point is a
weaker condition than for the geometries corresponding to
$|1|$--gradings we have treated so far. The difference is that for
higher gradings, the subalgebra $\frak p_+$ does not act trivially on
the representation $\lieg/\liep$, for which $TM$ is the associated bundle. Only the subalgebra $\lieg_2$ is trivial on $\lieg / \liep$, while
$\lieg_1$ actually injects into $L(\lieg_{-2},\lieg_{-1})$. Hence the
fact that an infinitesimal automorphism $\eta$ has a higher order fixed
point at $x_0\in M$ only implies that the differential
$D_{x_0}\phi_\eta^t:T_{x_0}M\to T_{x_0}M$ satisfies
$D_{x_0}\phi^t_\eta(\xi)-\xi\in T^{-1}_{x_0}M$ and restricts to the
identity map on the contact subspace $T^{-1}_{x_0}M$.

In order that $D_{x_0} \varphi_\eta^t = \id$,
the isotropy of $\eta$ must be contained in the subbundle of $T^*_{x_0}M$
corresponding to the $P$--invariant subspace
$\lieg_2\subset\liep_+$, which equals the annihilator of
the contact subspace $T^{-1}_{x_0}M$. For these isotropies, we can now
prove a uniform result for all parabolic contact structures.
The algebraic background needed to treat this case is a simple
standard result, which at the same time provides the key to the
classification of contact gradings:

\begin{lemma}\label{lem:contact}
  Let $\lieg$ be a real or complex simple Lie algebra endowed with a
  contact grading $\lieg=\lieg_{-2}\oplus\dots\oplus\lieg_2$. Then for
  any non--zero element $Z\in\lieg_2$ there exists $X\in\lieg_{-2}\cap
  T_{\lieg_-}(Z)$ such that $[Z,X]=A$ is the grading element of
  $\lieg$. In particular, for any representation $\Bbb V$
  corresponding to a harmonic curvature component of the geometry
  determined by $\lieg$, we have $\Bbb V_{st}(A)=\{0\}$.
\end{lemma}
\begin{proof}
  For the first part, see the proof of proposition 3.2.4 and section
  3.2.10 in \cite{cap.slovak.book.vol1}. The second part then follows
  since the curvature of a regular, normal, parabolic geometry is
  concentrated in positive homogeneities, hence so are all harmonic
  curvature components.
\end{proof}

The transition to geometry is also very simple in this case, thanks to the \textit{chains}, special distinguished curves defined for any parabolic contact structure; these generalize the well known chains for CR
structures, see \cite{csz.curves},
\cite{cz.chains}, and section 5.3.7 in
\cite{cap.slovak.book.vol1}.  Chains are uniquely
determined by their initial direction up to reparametrization and, from any point $x$, chains emanate in each direction transverse to the
contact hyperplane. 
They are defined as distinguished curves of
the form $\pi(\exp(b,sX))$ with $X\in\lieg_{-2}$ and $b \in B$.  For a fixed
$b$, there is only one such curve, up to parametrization, but varying $b$ along a fiber in the
Cartan bundle yields all chains emanating from $\pi(b)$.

\begin{theorem}\label{thm.contact}
  Let $M$ be a smooth manifold endowed with a parabolic contact
  structure and let $\eta\in\alginf(M)$. Suppose that $\eta$ has a higher
  order fixed point at $x_0\in M$ such that the isotropy $\al\in
  T^*_{x_0}M$ vanishes on the contact subbundle $T^{-1}_{x_0}M$. Then
  the higher order fixed point $x_0$ is smoothly isolated, and there is an open set $U$ with $x_0 \in \overline{U}$ on which the geometry is locally flat.
\end{theorem}

\begin{proof}
For all $b \in\pi^{-1}(x_0)$, the element $Z\in\frak p_+$
  corresponding to the isotropy $\al$ via $b$ lies in $\lieg_2$.  
  The element $X\in\lieg_{-2}\cap T_{\lieg_-}(Z)$ from lemma
  \ref{lem:contact} then gives rise to a distinguished curve $\sigma$, which is
  a chain. By corollary \ref{cor:sl2-path} any harmonic curvature
  component vanishes on a neighborhood $U$ of $\sigma \backslash \{ x_0 \}$, as desired.  
  
  It remains to show the centralizer $C_{\lieg_-} (Z)$ is trivial.  Denote by $B$ the Killing form of $\lieg$.  Then 
  $$ (Y,Y') \mapsto B(Z,[Y,Y']) = B([Z,Y],Y')$$
  is nondegenerate on $\lieg_{-1}$.  Thus $C_{\lieg_-}(Z) = 0 = C(\alpha)$, so $x_0$ is smoothly isolated by proposition \ref{prop.attractor} (3).
   \end{proof}
  
 

\begin{remark}
In the proof above, varying $b \in \pi^{-1}(x_0)$ gives vanishing of the harmonic curvature along all chains through $x_0$.  We expect that these form a dense subset of a neighborhood of $x_0$, so that the geometry is flat on a neighborhood of $x_0$; however, we are not aware of a general proof of this density of chains.  

In the case of CR structures (see section \ref{sect:CR}), it is
possible to conclude flatness on a neighborhood of $x_0$ thanks to the
Fefferman construction (see, for example, \cite{koch.cr}).  Given
$\eta$ as above with $D_{x_0} \varphi^t_\eta \equiv \id$, there is a
conformal pseudo-Riemannian structure on a neighborhood $V$ of $x_0$
product $S^1$ and a lift $\hat{\eta}$ of $\eta$ to a conformal Killing
field.  The chains $\gamma_X$ preserved by $\varphi^t_\eta$ lift to
null geodesics preserved by $\varphi^t_{\hat{\eta}}$, and the
$S^1$-fiber over $x_0$ is a pointwise fixed null geodesic.  This
conformal flow is in fact strongly essential at each point of $\{ x_0
\} \times S^1$; more precisely, it is a \emph{null translation} (see
\cite[theorem 4.3]{fm.champsconfs} or \cite[proposition
  3.7]{cap.me.proj.conf}).

By the results cited above, $V \times S^1$ is conformally flat on a
set of the form
$$ \hat{U} = \bigcup_{x \in \Delta} \mathcal{C}(x)$$
where $\Delta$ is the fiber over $x_0$, and $\mathcal{C}(x)$ is the
null cone through $x$.  Now it is easy to see that the projection of
$\hat{U}$ to $M$ is a neighborhood of $x_0$, and that it is CR-flat by the
Fefferman correspondence.
\end{remark}

\subsection{Partially integrable almost CR structures}
\label{sect:CR}
As mentioned before, these are the most important examples of parabolic
contact structures. Let $M$ be a connected manifold of odd dimension, say $2n+1$, endowed with a contact distribution
$H=T^{-1}M\subset TM$.   An equivalent formulation of the contact condition is that $H$ is locally equal to the kernel of a one--form $\la$ such that $\la\wedge
(\wedge^n d\la)$ is nowhere vanishing. Another equivalent description
of the maximal non--integrability of $H$ is that for each point $x\in M$, the
skew symmetric bilinear map $\Cal L_x:H_x\x H_x\to T_xM/H_x$ induced
by the Lie bracket of vector fields is non--degenerate.  

An almost CR
structure on $M$ consists of a contact distribution $H$ as above, together with a complex structure on $H$---that is, a smooth
bundle map $J:H\to H$ such that $J^2=J\o J=-\id$. To obtain a
parabolic contact structure, $J$ has to satisfy the following compatibility condition, called \textit{partial
  integrability}: $\Cal L_x(J_x(\xi),J_x(\eta))=\Cal
L_x(\xi,\eta)$ for all $x\in M$ and $\xi,\eta\in H_x$.  Assuming this property, $\Cal L_x$ is the imaginary part of a Hermitian form $\Cal L_x^{\BC}$ on $H_x$, uniquely determined up to scale, with
values in $(T_xM/H_x)\otimes\BC$, which is called the Levi--form at
$x$. The signature $(p,q)$, where $p\geq q$, of $\mathcal{L}_x^{\BC}$ is well-defined and constant on $M$.  When the Levi--forms are
definite, the almost CR structure is called \textit{strictly
  pseudoconvex}.  Strictly pseudoconvex almost CR structures are infinitesimally modeled on $S^{2n+1} \subset \BC^{n+1}$.

The main source of almost CR manifolds are real
hypersurfaces in complex manifolds. Indeed, if $M\subset N$ is such a hypersurface, for $N$ a complex manifold, then for each $x\in
M$, define $H_x\subset T_xM\subset T_x N$ to be the maximal
complex subspace contained in $T_xM$, so $H_x=T_xM\cap iT_xM$. The spaces $H_x$ form a
smooth distribution of corank one on $M$, which is generically contact. Then $J$ is the restriction of the complex
structure on $T_xN$ to $H_x$. Integrability
of the complex structure on $N$ implies partial
integrability of the resulting almost CR structure on $M$; in fact,
it satisfies a stronger condition, called
\textit{integrability}, and hence is called a \textit{CR--structure}. 

The general formulation of integrability is as
follows.  For a partially integrable almost CR structure
$(M,H,J)$, for any two sections $\xi,\eta\in\Ga(H)$, the difference $[\xi,\eta]-[J(\xi),J(\eta)]$ lies in $\Ga(H)$. Now define an
analog of the Nijenhuis tensor 
\begin{eqnarray*}
N & : & \Ga(H)\x\Ga(H)\to\Ga(H)  \\ 
& & (\xi,\eta)\mapsto [\xi,\eta]-[J(\xi),J(\eta)]+J([J(\xi),\eta]+[\xi,J(\eta)])
\end{eqnarray*}

It is bilinear over smooth functions and thus indeed induces a tensor
$N\in \Ga(\La^2H^*\otimes H)$.  A partially integrable almost CR
structure is \emph{integrable} if $N \equiv 0$.  Note that by
construction, the CR Nijenhuis tensor $N$ is of type $(0,2)$,
conjugate linear in both variables, meaning $N(J(\xi), \eta) = N(\xi,
J(\eta)) = -J(N(\xi,\eta))$.

The description of partially integrable almost CR structures of
signature $(p,q)$ as parabolic geometries uses the group
$G=PSU(p+1,q+1)$ and the parabolic subgroup $P< G$
defined as the stabilizer of a point in the projective space ${\bf CP}^{p+q+1}$ corresponding to an isotropic complex line in
$\BC^{p+1,q+1}$. Thus the homogeneous space $G/P$ can be identified
with the complex projectivization of the light cone in
$\BC^{p+1,q+1}$.  For more details, see section 4.2.4 of the book
\cite{cap.slovak.book.vol1}.

The $|2|$--grading on the Lie algebra $\frak{su}(p+1,q+1)$ determined by $P$ is given
by splitting matrices into blocks with respect to a basis starting
with a vector $v$ in the isotropic line stabilized by $P$, ending with a
complementary isotropic vector $w$, and having an orthonormal basis of $\mbox{span} \{ v,w \}^\perp$ in the middle.  These
matrices have the form
$$
\left\{
\begin{pmatrix} a & Z & iz \\ X & A-\tfrac{a-\bar a}{p+q}\id & -\Bbb I Z^*\\
  ix & -X^*\Bbb I & -\bar a \end{pmatrix} 
  \ : \ 
  \begin{array}{l}
a\in\BC \\
x,z\in\BR \\
X\in\BC^n, Z\in\BC^{n*} \\
A\in\frak{su}(p,q)
\end{array}
\right\}
$$
where $\Bbb I$ is the diagonal matrix $ \id_p \oplus (- \id_q)$. In the grading, $\lieg_{-2}$ is the subspace corresponding to $ix$,
$\lieg_{-1}$ to $X$, and so on. In particular, $\lieg_-$ and $\liep_+$
both are complex Heisenberg algebras of signature $(p,q)$, which
reflects the CR structure. The group $G_0$ is the conformal unitary
group of the Hermitian form of signature $(p,q)$ on $\lieg_{\pm 1}$
with the natural extension of the action to $\lieg_{\pm 2} \cong [\lieg_{\pm 1}, \lieg_{\pm 1}]$, and $P$
is $G_0 \ltimes \mbox{Heis}_{p,q}$.

For $Z \in \lieg_1$ and $X \in \lieg_{-1}$, the bracket $[Z,X] = (ZX,\Bbb{I} Z^* X^* \Bbb{I} - XZ) \in \lieg_0 \cong \BC \oplus \mathfrak{su}(p,q)$.  The further brackets we will want are
\begin{eqnarray*}
\ \ [[Z,X],X]  & = & - 2ZXX + X^* \Bbb{I} X \Bbb{I} Z^*   \\
\ \ [[Z,X],Z] & = & 2 ZXZ - Z \Bbb{I} Z^* X^* \Bbb{I}
 \end{eqnarray*}

Next we find the decomposition of
$\liep_+$ into $P$--orbits. The nonzero elements of $\lieg_2$ form one
$P$--orbit. Non--degeneracy of the bracket
$\lieg_1\x\lieg_1\to\lieg_2$ implies that any element in
$\liep_+$ with non--zero component in $\lieg_1$ is conjugate into $\lieg_1$ by an element of $P$.  On $\lieg_1$, the $P$--action
reduces to the standard action of the conformal unitary group, so
there are three $P$--orbits of elements not contained in $\lieg_2$,
corresponding to the sign of the inner product of the projection of
the element to $\liep_+/\lieg_2\cong\BC^{p,q}$.  The main
difference for our results is between null and non--null elements. We will refer to the isotropies at higher order fix points corresponding to these two cases as
\textit{null}, or \textit{non--null},  \emph{transversal
  isotropy}, respectively. The case of isotropy in $\lieg_2$ was already treated
in section \ref{sect:contact}. 

To formulate our results, it remains to describe the harmonic
curvature.  CR structures have one harmonic torsion and one harmonic
curvature. The harmonic torsion is the CR Nijenhuis tensor $N$ defined
above, and its vanishing, which by definition is equivalent to
integrability, is also equivalent to torsion freeness of the parabolic
geometry.  The harmonic curvature $\rho$ is a section of
$\La^2H^*\otimes H^*\otimes H$; it is a partially defined two--form of
type $(1,1)$ with values in skew--Hermitian endomorphisms of the
contact subbundle $H$.  It is totally trace free and has some
additional symmetries, which will not be relevant for our purpose
below.

Propositions \ref{prop.nonnull.cr} and
\ref{prop.null.cr} give a fairly detailed description of the flow around
a higher order fixed point with non--null or null isotropy, respectively.  Together with proposition \ref{thm.contact},
these imply:

\begin{theorem}
\label{thm.cr}
Let $M$ be a manifold with a partially integrable almost CR structure
and let $\eta \in \alginf(M)$ vanish to higher
order at $x_0 \in M$ with isotropy $\al\in T^*_{x_0}M$. Then 
\begin{enumerate}
\item The harmonic torsion vanishes at $x_0$.  If the structure is integrable, then the harmonic curvature vanishes at $x_0$.

\item If $\al$ annihilates $H_{x_0}\subset T_{x_0}M$ or if the
projection of $\al$ to $H^*_{x_0}$ is non--null, then there is an open
set $U \subset M$ with $x_0 \in \overline{U}$ on which the restriction
of the almost CR structure is flat; in particular, this conclusion holds if $M$ is strictly pseudoconvex.
\end{enumerate}
\end{theorem}

In the case that the CR structure on $M$ is real analytic and integrable,
Beloshapka \cite{beloshapka.str.ess.cr} and Loboda \cite{loboda.str.ess.cr}
proved that existence of nontrivial $\eta \in \mathfrak{inf}(M)$ vanishing
to higher order implies that M is flat.  Their proof involves calculations
with Taylor series in Moser normal coordinates.

\subsubsection{Non--null transverse isotropy}\label{sec.cr.nonnull}

As before, we start by collecting the algebraic background needed to
treat this case.
\begin{lemma}\label{lemma.cr.nonnull}
  Let $Z\in\lieg_1$ be non--isotropic, so $Z\Bbb
  IZ^*\neq 0$. Then 
  \begin{enumerate}
\item The sets associated to $Z$ are 
\begin{gather*}
 \{0\}=C_{\lieg_-}(Z)\subset F_{\lieg_-}(Z)=\{X\in\lieg_{-1}:ZX=X^*\Bbb
IX=0\}  \\
T_{\lieg_-}(Z) = \{ \tfrac{2}{Z\Bbb IZ^*}\Bbb IZ^*\}
\end{gather*}
 
 \item Let $X_0 = \frac{2}{Z\Bbb IZ^*}\Bbb IZ^*$, and let $A=[Z,X_0]$.  Then $A$ belongs to $\lieg_0$.  All eigenvalues of $A$ on $\lieg_-$ are
negative, and for any representation $\Bbb V$ corresponding to a
harmonic curvature or torsion component, $\Bbb V_{st}(A)=0$.
\end{enumerate}
\end{lemma}

\begin{proof}
  For any non--zero $ix\in\lieg_{-2}$ the map $W\mapsto
  [ix,W]$ is an injection $\lieg_1\to\lieg_{-1}$.  Because $Z \in \lieg_1$, the sets $F_{\lieg_-}(Z)$ and $C_{\lieg_-}(Z)$ are in $\lieg_{-1}$. Now let $X\in\lieg_{-1}$.  If $[Z,X]=0$, then $ZX=0$.  From $[[Z,X],X] = 0$, it follows that $\Bbb
  IZ^*$ and $X$ are linearly dependent. But if $X=\la\Bbb IZ^*$, then $ZX=\la Z\Bbb IZ^*$.  Since $Z$ is non--isotropic, $\la=0$, so $C_{\lieg_-}(Z)=0$.

For $X\in\lieg_{-1}$, grading considerations
imply $X\in F_{\lieg_-}(Z)$ if and only if $\ad(X)^2(Z)=0$.  If $X$ and
$\Bbb IZ^*$ are linearly independent, then $[[Z,X],X]$ vanishes if and only if
$ZX=\langle X,X\rangle=0$. If, on the other hand, $X=\la\Bbb IZ^*$ for $\la\in\Bbb C$,
vanishing of $[[Z,X],X]$ is equivalent to $|\la|^2=2\la^2$, which is
satisfied only for $\la=0$.  The description of $F_{\lieg_-}(Z)$
now follows.

Now suppose $X\in T_{\lieg_-}(Z)$.  Then $[[Z,X],Z] = 2Z$ implies $X^*\Bbb I=\la Z$, and 
$$2 \lambda Z \Bbb{I} Z^* - \lambda Z \Bbb{I} Z^* = 2ZX-\la
  Z\Bbb IZ^*=2$$
    The only possibility is thus $\lambda = \frac{2}{Z \Bbb{I} Z^*}$.  It is easy to verify that
  for $X_0 = \frac{2}{Z \Bbb{I} Z^*} \Bbb{I} Z^*$, we have $[[Z,X_0],X_0]=-2X_0$, so $T_{\lieg_-}(Z)=\{X_0\}$, as claimed. 
  
  For (2), compute $ZX_0=2$ and $\Bbb IZ^*X_0^*\Bbb I-X_0Z=0$, whence
  $A=[Z,X_0]$ is the block diagonal matrix with entries $(2,0, \ldots, 0,-2),$
  which is twice the grading element. The statement on $\Bbb
  V_{st}(A)$ then immediately follows from homogeneity considerations
  as in the proof of lemma \ref{lem:contact}.
\end{proof}

Here is the resulting proposition for non--null transverse isotropy.

\begin{proposition}
\label{prop.nonnull.cr}
Let $M$ be a manifold endowed with a partially
integrable almost CR structure and let $\eta \in \alginf(M)$ vanish to higher order at $x_0 \in
M$. Assume further that the isotropy $\al\in T_{x_0}^*M$ of $\eta$ restricts to a 
non--null element of $(H_{x_0})^*$.  Then
\begin{enumerate}
\item The sets $C(\alpha) \subset F(\alpha)$ in $T_{x_0} M$ are  
$$\{ 0 \} = C(\alpha) \subset F(\al)=\{\xi\in H_{x_0}:\Cal L_{x_0}(\xi,J(\xi))=0, \ \al(\xi)=0\} $$ 
The real line spanned by $T(\alpha)$ consists of all elements dual to real multiples of $\left. \alpha \right|_{H_{x_0}}$.

\item Let $b_0 \in \pi^{-1}(x_0)$ be a point such that $\alpha$ corresponds via $b_0$ to an element of $\lieg_1$.  Then in the normal coordinate chart determined by $b_0$, there is an open neighborhood $U$ of $0$ such that:
\begin{enumerate}

\item The higher order fixed point $x_0$ is smoothly isolated in the strongly fixed set.  Elements of $U\cap F(\al)$ are zeros of $\eta$.

\item Let
$$\ell = \{ \xi \in  H_{x_0}  \ : \ \ker(\Cal L_{x_0}(\xi, J(\cdot) ))=\ker(\al)\cap H_{x_0} \}$$
 Then $\ell$ is a line, and on $\ell \cap U$, the flow acts by $\xi\mapsto
\tfrac{1}{1+t\al(\xi)}\xi$ whenever $t \alpha(\xi) > 0$.
\end{enumerate}
\item There is an open neighborhood of
$\ell\setminus\{0\}$ in $U$ on which the geometry is locally
flat.  
\end{enumerate}
\end{proposition}

\begin{proof}
We fix $b_0$ as in (2) and let $Z\in\lieg_1$ be the element
corresponding to $\al$ via $b_0$. Then $C_{\lieg_-}(Z)=\{0\}$ by lemma \ref{lemma.cr.nonnull} and as in the proof of proposition
\ref{prop.1-graded}, this implies that $C(\alpha) = 0$. 
Next, lemma \ref{lemma.cr.nonnull} implies the description of $F(\al)$,
and it shows that $T(\al)$ consists of the unique element
$\xi_0\in\ell$ such that $\al(\xi_0)=2$, so (1) is proved. Then (2) follows immediately
from propositions \ref{prop.attractor}, \ref{prop.sl2.triple} (with
analogous simplifications as in proposition
\ref{prop.simplifications}) and \ref{prop.ext.holonomy}.
\end{proof}

\subsubsection{Null transverse isotropy}\label{sec.cr.null}
In this case, we cannot deduce a local flatness result, so we will
focus on proving vanishing of the curvature on a complex curve of
higher order fixed points through $x_0$ and providing the algebraic
background for a more detailed analysis. It is possible to deduce
restrictions on torsion and curvature outside of the strongly fixed
set as was done for almost Grassmannian structures in part (4) of
proposition \ref{prop.rank1.grassmannian}.  The CR statements are
rather complicated to formulate, so we do not prove them in detail
here.

\begin{lemma}\label{lemma.cr.null}
  Let $Z\in\lieg_1$ be isotropic, so $Z\Bbb
  IZ^*=0$. Then 
\begin{enumerate}
\item The sets associated to $Z$ are
\begin{gather*}
\BC \cdot\Bbb IZ^*=C_{\lieg_-}(Z)\subset F_{\lieg_-}(Z)=
\{X\in\lieg_{-1}:ZX=X^*\Bbb IX=0\} \\
T_{\lieg_-}(Z)=\{X\in\lieg_{-1}:ZX=1, \ X^*\Bbb IX=0\}
\end{gather*}

\item For $X\in T_{\lieg_-}(Z)$, let $A=[Z,X] \in \lieg_0$. Then $A$ has
non--positive eigenvalues on $\lieg_-$, and the $0$-eigenspace equals $C_{\lieg_-}(Z)$. 
For $\Bbb V\subset \La^{(0,2)}\lieg_1\otimes\lieg_{-1}$ the representation corresponding to the harmonic torsion, $\Bbb V_{st}(A)\subset
\La^2\lieg_1\otimes X^\perp$, 
and 
$$ \Bbb V_{ss}(A)\subset (\BC\cdot \Bbb IX^*)\wedge (\ker(X) \cap
Z^\perp)\otimes\BC X $$
Let $\Bbb U\subset
\La^{(1,1)}\lieg_1\otimes\lieg_0$ be the representation corresponding to the harmonic
curvature, and let $\Phi\in \Bbb U_{st}(A)$.  If 
$$\Phi(Y_1,Y_2).Z=0 \qquad \forall \ Y_1,Y_2\in\lieg_{-1}$$
then $\Phi(Y_1,Y_2)=0$ whenever $Y_1\in\BC X$ and $Y_2\in\lieg_{-1}$.
\end{enumerate}
\end{lemma}

\begin{proof}
  As in the proof of lemma \ref{lemma.cr.nonnull}, $F_{\lieg_-}(Z)\subset\lieg_{-1}$, and $X\in\lieg_{-1}$ lies in
  $F_{\lieg_-}(Z)$ if and only if $\ad(X)^2(Z)=0$. As before, $[Z,X]=0$ implies $X=\la\Bbb IZ^*$ for some
  $\la\in\BC$.  Conversely, any complex multiple of $\Bbb I Z^*$ commutes with $Z$, so
  $C_{\lieg_-}(Z)$ is as claimed in (1). Next suppose $[[Z,X],X]=0$.  Then $X$ and $\Bbb I Z$ are linearly dependent, in which case $X$ is in $C_{\lieg_-}(Z)$, or $ZX = 0 = X^*\Bbb IX$.  The description of
  $F_{\lieg_-}(Z)$ is now proved.

  For $X\in T_{\lieg_-}(Z)$, grading considerations imply that
  $X\in\lieg_{-1}$, and then $X$ is linearly independent from
  $\Bbb IZ^*$, since otherwise $[Z,X]=0$.  But then $[[Z,X],X]=-2X$ is equivalent to $ZX=1$ and $X^*\Bbb IX=0$.  These two conditions also imply that
  $[[Z,X],Z]=2Z$, so the description of $T_{\lieg_-}(Z)$ follows. 
  
  Now, for
  $X\in T_{\lieg_-}(Z)$, let $A=[Z,X]$ as in (2). The $A$-action on $\lieg_{-2}$ is
  multiplication by $-2ZX=-2$, while the action on $\lieg_{-1}$ maps
  $Y$ to $-Y-ZYX+X^*\Bbb IY \Bbb IZ^*$. Then $X \in \lieg_{-1}$ spans the $-2$-eigenspace of $A$ in $\lieg_{-1}$; $\ker(Z)\cap X^\perp$ is the $-1$-eigenspace; and 
  $\Bbb IZ^*$ spans the $0$-eigenspace. Thus all eigenvalues of $A$ on $\lieg_-$ are nonpositive, and the $0$-eigenspace is $C_{\lieg_-}(Z)$, as stated in (2).
  
  The eigenspace structure on $\frak p_+$
  follows readily by dualization:
  \begin{center}
\begin{tabular}{| c | c |}
\hline
{\bf subspace} & {\bf eigenvalue} \\
\hline
$\BC \cdot\Bbb IX^*$ & $0$ \\
$\ker(X) \cap Z^\perp$ & $1$ \\
$\lieg_{2} \oplus \BC Z$ & $2$ \\ 
\hline
\end{tabular}
\end{center}

Now we can analyze the eigenspace structure on $\Bbb V$ and
$\Bbb U$. On $\La^{(0,2)}\lieg_1$, the eigenvalues of $A$ range from
$1$ to $3$. (The eigenspaces with eigenvalues $0$ and $4$ which are
present in $\La^2\lieg_1$ both come from the alternating square of a
complex line, so they both are contained in $\La^{(1,1)}\lieg_1$.)
Then the values of a map in $\Bbb V_{st}(A)$ lie in the sum of the eigenspaces in $\lieg_{-1}$ with
negative eigenvalues. This sum is exactly $X^\perp\subset\lieg_{-1}$,
which proves the claim on $\Bbb V_{st}(A)$. 

On the other hand, $\Bbb
V_{ss}(A)$ evidently coincides with the $(-1)$--eigenspace of $A$ in
$\Bbb V$, which is the space claimed in the proposition.

The eigenvalues of $A$ on $\lieg_0 \subset \lieg_1 \otimes \lieg_{-1}$
range from $-2$ to $2$, so the possible eigenvalues on $\Bbb U$ range
from $-2$ to $6$. Since $Z$ spans the $(+2)$--eigenspace in $\lieg_1$,
the condition that $\Phi(Y_1,Y_2).Z=0$ implies $\Phi(Y_1,Y_2)$ is in
the sum of $A$-eigenspaces of $\lieg_0$ with eigenvalues $\geq -1$.
If this holds for all $Y_1,Y_2\in\lieg_{-1}$, and if $\Phi\in \Bbb
U_{st}(A)$, then the $\La^{(1,1)}\lieg_1$-components of $\Phi$ lie in
the sum of $0$ and $1$-eigenspaces.  This sum in turn is made up of
components from the $0$ and $1$-eigenspaces in $\lieg_1$, which form
the annihilator of $X$.  Now the last claim of (2) follows.
  \end{proof}

\begin{proposition}\label{prop.null.cr}
  Let $M$ be a manifold endowed with a nondegenerate, partially
  integrable almost CR structure, and let $\eta \in \alginf(M)$ vanish to higher order at $x_0 \in
  M$. Assume further that the isotropy $\al\in T_{x_0}^*M$ of $\eta$
  restricts to a nonzero, null element of $(H_{x_0})^*$. 
\begin{enumerate} 
\item The set $C(\alpha)$ is the complex line spanned by the elements dual to $\al|_{H_{x_0}}$, while 
$$F(\al)  = \{ \xi\in H_{x_0}\ : \ \al(\xi)=0, \ \Cal L(\xi,J\xi)=0 \}$$
and
$$T(\al) = \{ \xi \in H_{x_0} \ : \ \al(\xi) = 1, \ \Cal L(\xi,J\xi)=0 \}$$

\item  Let $b_0 \in \pi^{-1}(x_0)$ be a point such that $\alpha$ corresponds via $b_0$ to an element of $\lieg_1$.  Then in the normal coordinate chart determined by $b_0$, there is an open neighborhood $U$ of $0$ such that
\begin{enumerate} 
\item Points of $C(\alpha) \cap U$ are higher order fixed points with transverse null isotropy.  Any $\xi \in F(\alpha) \cap U$ is a zero of $\eta$.
  
\item On the cone
$$ S = \{ \xi \in H_{x_0} \ : \ \alpha(\xi) \neq 0, \ \Cal L(\xi, J \xi) = 0 \}$$
the action of the flow along $\eta$ in normal coordinates is
$$\varphi_\eta^t(\xi) = \tfrac{1}{1+t\al(\xi)} \cdot \xi \ \mbox{for} \ t \alpha(\xi) > 0$$
\end{enumerate}

\item The Nijenhuis tensor vanishes on the complex curve of higher order fixed points $C(\alpha) \cap U$ in (1). If the structure is integrable, that is, CR, then the harmonic curvature also vanishes along the complex curve $C(\alpha) \cap U$.
\end{enumerate}
\end{proposition} 

\begin{Pf}
The descriptions of $C(\al)$, $F(\al)$ and $T(\al)$ in (1) follow immediately
from lemma \ref{lemma.cr.null} and then (2)(a) and (b) follow immediately from
propositions \ref{prop.attractor} and \ref{prop.sl2.triple}. 

For (3), let $b_0 \in \pi^{-1}(x_0)$ be such that $Z\in\lieg_1$
is the element corresponding to $\al$ via $b_0$.  Choose $X\in
T_{\lieg_-}(Z)$ and let $A=[Z,X]\in\lieg_0$. Then by proposition
\ref{prop.attr.curv}, the functions corresponding to the harmonic
torsion and the harmonic curvature must map $b_0$ to $\Bbb
V_{st}(A)$ and to $\Bbb U_{st}(A)$, respectively. Now from lemma
\ref{lemma.cr.null}, $\Bbb V_{st}(A)$ consists of maps
having values in $X^\perp$. But for a fixed $Z$, elements $X \in T_{\lieg_-}(Z)$ span $\lieg_{-1}$ over $\BC$.  Thus the harmonic torsion vanishes at $x_0$. This torsion is a scalar multiple of the Nijenhuis
tensor, which then also vanishes.

To complete the argument in the torsion free case, we proceed as in
the proof of proposition \ref{prop.rank1.grassmannian}. Denote by
$\ka:B\to \La^2\lieg_-^*\otimes\lieg$ the Cartan curvature function.  Again from section 2.3 of \cite{cap.srni04}, 
$[\ka_{b_0}(Y_1,Y_2),Z]=0$ for all $Y_1,Y_2\in\lieg_-$. If the
structure is CR, then the Cartan connection is torsion
free; the value $\ka_{b_0}$ is
homogeneous of degrees $\geq 2$, and the lowest non--zero homogeneous
component is harmonic. Hence taking $Y_1,Y_2\in\lieg_{-1}$, the
component in $\lieg_1$ of $[\ka_{b_0}(Y_1,Y_2),Z]$ coincides with the
value of $[\rho_{b_0}(Y_1,Y_2),Z]$, where $\rho$ is the function
corresponding to the harmonic curvature. By lemma \ref{lemma.cr.null},
$[\rho_{b_0}(Y_1,Y_2),Z]=0$ together with $\rho(b_0)\in\Bbb U_{st}(A)$
implies $\rho_{b_0}(Y_1,Y_2)=0$ for $Y_1\in\BC X$ and
$Y_2\in\lieg_{-1}$. Again, because $T_{\lieg_-}(Z)$ spans $\lieg_{-1}$ over $\BC$, the harmonic curvature vanishes at $x_0$, and along the strongly fixed component of $x_0$.
\end{Pf}

Last, we can easily deduce the following local version of the
Schoen-Webster theorem \cite{webster.cr.flow}, \cite{schoen.cr}, which
generalizes the result of \cite{frances.riemannien} from the conformal
Riemannian setting. Versions of this result were obtained by Vitushkin
under the additional assumption of integrability in
\cite{vitushkin.local.schoen} and by Kruzhilin for smooth, integrable
CR structures in \cite{kruzhilin.local.schoen}.

\begin{theorem}
\label{thm.strict.cr}
Let $M^{2n+1}$ be a connected, real-analytic, partially integrable, strictly pseudoconvex CR manifold, and let $\eta \in \alginf(M)$ vanish at $x_0 \in M$.  Then 
\begin{itemize}
\item  There is a neighborhood $U$ of $x_0$ invariant by the flow, on which $\{ \varphi^t_\eta \}$ is bounded and linearizable; or
\item  $M$ is spherical---that is, locally flat as a CR structure.
\end{itemize}
\end{theorem}

\begin{Pf}
Pick $b_0 \in \pi^{-1}(x_0)$, and let $Z \in \liep$ be the isotropy of $\eta$ at $x_0$ with respect to $b_0$.  Because $M$ is real-analytic, the full isotropy subalgebra at $x_0$ is algebraic; more precisely, the isotropies with respect to $b_0$ of all local Killing fields vanishing at $x_0$ form an algebraic subalgebra of $\liep$.  This fact is a consequence of the so-called Frobenius Theorem of \cite[sec 3.4]{gromov.rgs} and \cite[thm 3.11]{me.frobenius}.  An algebraic subalgebra is closed under Jordan decomposition (see \cite[4.3.4]{morris.ratner.book}), so we may assume $Z$ is either nilpotent or semisimple.

Recall $\liep \cong (\BC \times \mathfrak{su}(n)) \ltimes \mathfrak{u}$, where $\mathfrak{u}$ is the nilpotent radical, isomorphic to a $(2n+1)$-dimensional Heisenberg Lie algebra.  If $Z \in \mathfrak{u}$, then $M$ is flat on a nonempty open set by theorem \ref{thm.cr}, and so $M$ is flat everywhere because it is analytic.

Now suppose $Z$ is semisimple, so it is conjugate into $\lieg_0 \cong \BC \times \mathfrak{su}(n)$.  If $Z$ lies in a maximal compact subalgebra of $\lieg_0$, then $\ad(Z)$ preserves the subspace $\lieg_-$ and a positive-definite inner product on it.  This inner product descends to $T_{x_0} M$, where it is invariant by $\{ D_{x_0} \varphi^t_\eta \}$, the differential of the flow, and it pulls back via the normal coordinate chart determined by $b_0$ to a Riemannian metric on a precompact neighborhood $U$ of $x_0$, which is invariant by $\{ \varphi^t_\eta \}$.  It follows that the flow is bounded and linearizable on $U$.

Now suppose that $Z$ generates an unbounded 1-parameter subgroup, so it can be written $Z = A + K$, where $A$ is a nonzero multiple of the grading element, and $K \in \mathfrak{su}(1) \times \mathfrak{su}(n)$. First compute that for $X \in \lieg_-$,
$$ e^{tZ} e^{sX} = e^{sa^t k_t(X)} e^{tZ}$$
where $\Ad(e^{tZ}) = \Ad(e^{tA}) \circ \Ad(e^{tK}) = a^t \cdot k_t$, for $a \neq 0$ and $k_t \in \mbox{SU}(1) \times \mbox{SU}(n)$.  We can assume that $a <1$, by replacing $\eta$ with $- \eta$ if necessary.  Now proposition \ref{prop.basic.holonomy} impies that all the distinguished curves emanating from $x_0$ tend to $x_0$ under the flow by $\varphi^t_\eta$, and that $e^{tZ}$ is again a holonomy path with attractor $b_0$ along the curves $\exp_{b_0}(sX)$.  Now $e^{tZ}$ satisfies the hypotheses of proposition \ref{prop.ext.holonomy} for all $Y \in \lieg_-$, so it is also a holonomy path at points projecting onto a neighborhood of $\exp_{b_0}(sX)$.  The action of $e^{tZ}$ on the representations for the harmonic curvature and torsion is expanding---that is, any vector $w$ in one of these representations satisfies $||e^{tZ}(w)|| \rightarrow \infty$ as $t \rightarrow \infty$.  By arguments similar to those in proposition \ref{prop.curv.stab}, both the torsion and curvature vanish wherever we have $e^{tZ}$ as a holonomy path.  Thus the curvature vanishes on an open set, so it vanishes everywhere by analyticity.
\end{Pf}

\bibliographystyle{ieeetr.bst}
\bibliography{karinsrefs}

\begin{tabular}{lll}
Andreas \v{C}ap  & \qquad & Karin Melnick \\
Faculty of Mathematics & \qquad & Department of Mathematics \\
University of Vienna & \qquad & University of Maryland \\
Oskar--Morgenstern--Platz 1 & \qquad & College Park, MD 20742 \\
1090 Vienna, Austria & \qquad & USA \\
andreas.cap@univie.ac.at & \qquad & karin@math.umd.edu
\end{tabular}

\end{document}